\documentclass[11pt, twoside, a4paper, reqno]{amsart}

\usepackage{amsfonts, amssymb, amscd, amsmath, amsthm}
\usepackage[utf8]{inputenc}
\usepackage{graphicx}
\usepackage{float}
\usepackage{multicol}
\allowdisplaybreaks 
\usepackage{mathtools}
\usepackage{amsmath, amssymb}
\usepackage{amscd}
\usepackage{xcolor}
\usepackage{tikz-cd} 
\usepackage{graphicx} 
\usepackage[charter]{mathdesign}
\usepackage{enumitem}

\allowdisplaybreaks 
\usepackage{mathtools}
\usepackage[bookmarksnumbered, colorlinks, plainpages]{hyperref}

\usepackage{comment}
\newtheorem{definition}{Definition}[section]
\newtheorem{theorem}[definition]{Theorem}
\newtheorem{lemma}[definition]{Lemma}

\newtheorem{proposition}[definition]{Proposition}
\theoremstyle{definition}
\newtheorem{remark}[definition]{Remark}
\newtheorem{note}[definition]{Note}
\newtheorem{example}[definition]{Example}

\newcommand{\diX}{\int^\oplus_{X}}
\newcommand{\dilX}{\int^{\oplus_\text{loc}}_{X_{\text{loc}}}}
\newcommand{\dmu}{\mathrm{d} \mu (p)}
\newcommand{\la}{\left\langle}
\newcommand{\ra}{\right\rangle}
\newcommand\up[1]{\mbox{\raisebox{1pt}{\ensuremath{#1}}}} 

\title[Direct Integral of Locally Hilbert Spaces]{Direct Integral of Locally Hilbert Spaces}
\author[Kulkarni]{Chaitanya J. Kulkarni}
\address{Chaitanya J. Kulkarni, Indian Institute of Science  Education and Research (IISER) Mohali, Knowledge City, S.A.S Nagar, Punjab 140306, India.}
\email{chaitanyakulkarni58@gmail.com}

\author[Pamula]{Santhosh Kumar Pamula}
\address{Santhosh Kumar Pamula, Indian Institute of Science  Education and Research (IISER) Mohali, Knowledge City, S.A.S Nagar, Punjab 140306, India.}
\email{santhoshkp@iisermohali.ac.in}

\date{}
\begin{document}

\maketitle

\begin{abstract}
In this work, we introduce the concept of the direct integral of locally Hilbert spaces. This notion is formulated such that the direct integral of locally Hilbert spaces forms a locally Hilbert space. We then define two classes of locally bounded operators on the direct integral of locally Hilbert spaces namely the class of decomposable locally bounded operators and the class of diagonalizable locally bounded operators.  We prove that the set of all decomposable locally bounded operators forms a locally von Neumann algebra, while the set of diagonalizable locally bounded operators forms an abelian von Neumann algebra, and we show that they are commutant of each other.
   

\end{abstract}


\section{Introduction and Preliminaries} \label{sec;Introduction and Preliminaries}

The concept of the direct integral of Hilbert spaces generalizes the notion of direct sum of Hilbert spaces. In the context of direct integrals, the discrete index set used in direct sum is replaced by a suitable measure space. The notion of the direct integral of Hilbert spaces is associated with an abelian von Neumann algebra, referred to as the algebra of diagonalizable operators. The concept of direct integral of Hilbert spaces is crucial in the study of a representation of a separable $C^\ast$-algebra by using irreducible representations, or in the study of a von Neumann algebra by using factors (see \cite{OB1, DixC, DixV, KR2, Tak1,  Wils2, Wils1} for more details on this topic). In \cite{BK1, BK2}, authors used the concept of direct integral of Hilbert spaces to study the collection of unital completely positive maps from a $C^\ast$-algebra $\mathcal{A}$ to $\mathcal{B}(\mathcal{H})$, the algebra of bounded operators on a Hilbert space $\mathcal{H}$.

The notion of a locally Hilbert space is defined via the inductive limit of some strictly inductive system of Hilbert spaces (see Definition \ref{def;lhs}). In this context, one may consider a strictly inductive system of direct integral of Hilbert spaces. This leads to a fundamental question: When taking the inductive limit can we interchange the inductive limit and direct integral? In this article, motivated by this fundamental question, we extend the concept of direct integrals to the setting of locally Hilbert spaces. We begin by proposing a definition of the direct integral of locally Hilbert spaces using the theory of inductive limits of strictly inductive systems of direct integral of Hilbert spaces (Definition \ref{Defn: directint_loc}). This definition generalizes the concept of the direct sum of locally Hilbert spaces. Moreover, in the special case where each locally Hilbert space is actually a Hilbert space, the notion of the direct integral of locally Hilbert spaces coincides with the classical definition of the direct integral of Hilbert spaces (see Remark \ref{rem; examples of dilhs particular cases}). A direct integral of locally Hilbert spaces is defined such that the resulting space is again a locally Hilbert space. This naturally leads to define subcollections of locally bounded operators that adhere to the structure of direct integral. So, in this framework, we introduce the classes of decomposable and diagonalizable locally bounded operators on direct integral of locally Hilbert spaces. We show that the class of decomposable locally bounded operators is a locally von Neumann algebra, which we see by realizing it as the projective limit of a projective system of von Neumann algebras. However, the class of diagonalizable locally bounded operators turns out to be an abelian von Neumann algebra. Moreover, under suitable assumptions, we prove that these two classes of operators are commutants of each other. Throughout this article, we use the concepts of inductive limit and projective limit of locally convex spaces along with the framework of direct integrals. For this, we follow the results presented in \cite{BGP, AD, AG, AI, MJ3} , as well as the results on the theory of locally von Neumann algebras discussed in \cite{MF, MJ1, MJ2}.

The article is organized into three sections. In Section  \ref{sec;Introduction and Preliminaries}, we review key definitions and results from the theory of direct integrals of Hilbert spaces, as well as from the theory of locally $C^\ast$-algebra and locally von Neumann algebras. In Section \ref{sec; Direct integrals}, we introduce the concept of the direct integral of locally Hilbert spaces. We prove that this concept preserves the underlying quantized domain structure (see Proposition \ref{prop; QD}). Further, we show that in the case of finite direct sum and in the case when each of the locally Hilbert space is a Hilbert space the notion of direct integral of locally Hilbert spaces coincides with the well established notions in the literature (refer Remark \ref{rem; examples of dilhs particular cases}). We also provide some examples of the direct integral of locally Hilbert spaces. In Section \ref{sec; Decomposable and Diagonalizable Locally Bounded Operators}, we define and study decomposable and diagonalizable locally bounded operators. We show that the space of diagonalizable locally bounded operators forms an abelian von Neumann algebra, while the space of decomposable locally bounded operators forms a locally von Neumann algebra. Finally, we examine the relationship between these two classes and prove that under certain assumptions they are commutant of each other.

\subsection{Direct integral of Hilbert spaces}
We recall a few definitions and results from the theory of direct integral of Hilbert spaces. The reader is directed to \cite{OB1, DixC, DixV, KR2, Tak1} for a comprehensive reading of this topic. Throughout this article, for terminology and notations related to the theory of direct integral of Hilbert spaces, we refer to Chapter 14 of \cite{KR2}. 

\begin{definition}\cite[Definition 14.1.1]{KR2} \label{def;dihs}
If $X$ is a $\sigma$-compact locally compact (Borel measure) space, $\mu$ is the completion of a (positive) Borel measure on $X$, and $\{ \mathcal{H}_p \}_{p \in X}$ is a family of separable Hilbert spaces indexed by points $p$ in $X$, we say that a separable Hilbert space $\mathcal{H}$ is the direct integral of $\{ \mathcal{H}_p \}$ over $(X, \mu)$ \Big(we write: $\mathcal{H} = \diX \mathcal{H}_p \, \dmu$ \Big) when, to each $x \in \mathcal{H}$, there corresponds a function $ p \mapsto x(p)$ on $X$ such that $x(p) \in \mathcal{H}_p$ for each $p$ and 
\begin{enumerate}
\item $p \mapsto \la x(p), y(p) \ra$ is $\mu$-integrable, when $x, y \in \mathcal{H}$ and
\begin{align*}
\la x, y \ra = \int_X \la x(p), y(p) \ra \, \dmu
\end{align*}
\item if $x_p \in \mathcal{H}_p$ for all $p$ in $X$ and $p \mapsto \la x_p, y(p) \ra$ is integrable for each $y \in \mathcal{H}$, then there is a $x \in \mathcal{H}$ such that 
$x(p) = x_p$ $\mu$-a.e. 
\end{enumerate}
We say that $\diX \mathcal{H}_p \, \dmu$ and $p \mapsto x(p)$ are the (direct integral) decompositions of $\mathcal{H}$ and $x$ respectively.
\end{definition}

\begin{note} \label{note; measure space}
Throughout this article, the notation $(X, \mu)$ will always refer to a $\sigma$-compact, locally compact space $X$ with a measure $\mu$, where $\mu$ is the completion of a positive Borel measure on $X$, unless stated otherwise. From now onwards, we call $(X, \mu)$ just as a measure space, without mentioning these properties.
\end{note}

Given a measure space $(X, \mu)$, we denote by $\text{L}^\infty(X, \mu)$ the space of all essentially bounded measurable functions from $X$ to $\mathbb{C}$. Next, we recall the definition of a decomposable and diagonalizable bounded operators on the direct integral of Hilbert spaces. 
\begin{definition}\cite[Definition 14.1.6]{KR2} \label{def;Debo}
Let $(X, \mu)$ be a measure space, and let  $\{ \mathcal{H}_p \}_{p \in X}$ be a family of separable Hilbert spaces with $\mathcal{H} = \diX \mathcal{H}_p \, \dmu$.
\begin{enumerate}
    \item \label{def;Decbo} An operator $T$ in $\mathcal{B}(\mathcal{H})$ is said to be {\it decomposable}, if there is a family  $\{ T_p \in \mathcal{B}(\mathcal{H}_p) \}_{p \in X}$ such that for each $x \in \mathcal{H}$, we have 
\begin{align*}
(Tx)(p) = T_p x(p) \; \; \; \text{for a.e.} \; \; p \in X.
\end{align*}
Subsequently, $T$ is denoted by $\int^\oplus_X T_p \, \dmu$. Moreover, the norm of $T$ is given by
\begin{equation} \label{eq;norm of T}
\| T \| = \text{ess sup} \big \{ \| T_p \| \; : \; p \in X \big \}.
\end{equation}
\item \label{def;Diagbo} An operator $T$ in $\mathcal{B}(\mathcal{H})$ is said to be {\it diagonalizable}, if $T$ is decomposable and there exists a function $f \in \text{L}^\infty(X, \mu)$ such that for each $x \in \mathcal{H}$, we have 
\begin{align*}
(Tx)(p) = f(p) x(p) \; \; \; \text{ for a.e.} \; \; p \in X.
\end{align*}
\end{enumerate} 
\end{definition}

The following theorem provides a structure on the set of all decomposable and the set of all diagonalizable operators, and describes the relationship between them. 

\begin{theorem}\cite[Theorem 14.1.10]{KR2} \label{thm;DeDibo vNA}
Let $\mathcal{H} = \diX \mathcal{H}_p \, \dmu$  be as in Definition \ref{def;Debo}. Then the set of all decomposable operators is a von Neumann algebra with the abelian commutant coinciding with the set of all diagonalizable operators.  
\end{theorem}

In this work, our aim is to define suitable notions like direct integrals, decomposable operators, diagonaliazable operators, etc in locally Hilbert space setting. Before that, we turn our attention towards basic notations, terminologies and necessary concepts from the theory of locally Hilbert spaces, locally C*-algebras and locally von Neumann algebras. The following definitions are mainly drawn from \cite{AG, MJ1}, and for a detailed discussion on these topics, one can see \cite{BGP, AD, MF, AG, AG2, AG3, AI, MJ1, MJ2, MJ3, NCP}.

\subsection{Locally Hilbert space} 
A locally Hilbert space is the inductive limit of a strictly inductive system (or an upward filtered family) of Hilbert spaces. The formal definition is given below:

\begin{definition}\cite[Subsection 1.3]{AG} \label{def;sis}
Let $( \mathcal{H}_\alpha,\; \la \cdot, \cdot \ra_{\mathcal{H}_\alpha} )_{\alpha \in \Lambda}$ be a net of Hilbert spaces. Then $\mathcal{F} = \{\mathcal{H}_\alpha \}_{\alpha \in \Lambda}$ is said to be a strictly inductive system (or an upward filtered family) of Hilbert spaces if:
\begin{enumerate}
\item $(\Lambda, \leq)$ is a directed partially ordered set (poset);
\item for each $\alpha, \beta \in \Lambda$ with $\alpha \leq \beta \in \Lambda$ we have $\mathcal{H}_\alpha \subseteq \mathcal{H}_\beta$;
\item for each $\alpha, \beta \in \Lambda$ with $\alpha \leq \beta \in \Lambda$ the inclusion map $J_{\beta, \alpha} : \mathcal{H}_\alpha \rightarrow \mathcal{H}_\beta$ is isometric, that is,
$\la u, v \ra_{\mathcal{H}_\alpha} = \la u, v \ra_{\mathcal{H}_\beta}$ for all $u, v \in \mathcal{H}_\alpha$.
\end{enumerate}
\end{definition}

It is evident from Equation (1.13) of \cite{AG} that for a strictly inductive system $\mathcal{F} = \{\mathcal{H}_\alpha \}_{\alpha \in \Lambda}$ of Hilbert spaces, the inductive limit denoted by $\varinjlim\limits_{\alpha \in \Lambda} \mathcal{H}_\alpha$ exists, and it is given by
\begin{equation} \label{eq;lhs}
\varinjlim_{\alpha \in \Lambda} \mathcal{H}_\alpha = \bigcup_{\alpha \in \Lambda} \mathcal{H}_\alpha.
\end{equation}

\begin{definition}\cite[Subsection 1.3]{AG} \label{def;lhs}
A locally Hilbert space $\mathcal{D}$ is defined as the inductive limit of some strictly inductive system $\mathcal{F} = \{\mathcal{H}_\alpha \}_{\alpha \in \Lambda}$ of Hilbert spaces. 
\end{definition}

It is worth to point out that in the work of \cite{BGP, AD}, the authors used the term ``quantized domain" in place of locally Hilbert space. In particular, the article \cite{BGP} adopts the notation $\{\mathcal{H}; \mathcal{F}; \mathcal{D} \}$ to represent a quantized domain, where the Hilbert space $\mathcal{H}$ is the completion of the locally Hilbert space $\mathcal{D}$ given by the strictly inductive system $\mathcal{F}$. Whereas, in the work of \cite{MJ3}, the author used the notation $\mathcal{D}_\mathcal{F}$. However, throughout this work, we stick to the notation $\big \{ \mathcal{H}; \mathcal{F} = \{\mathcal{H}_\alpha\}_{\alpha \in \Lambda}; \mathcal{D} \big \}$ to represent a quantized domain, where the Hilbert space $\mathcal{H}$ is obtained by the completion of the locally Hilbert space $\mathcal{D}$ given by the strictly inductive system $\mathcal{F} = \{\mathcal{H}_\alpha\}_{\alpha \in \Lambda}$. Thus, we have 
\begin{equation*}
\mathcal{D} = \varinjlim_{\alpha \in \Lambda} \mathcal{H}_\alpha = \bigcup_{\alpha \in \Lambda} \mathcal{H}_\alpha \; \; \; \; \text{and} \; \; \; \; \overline{\mathcal{D}} = \mathcal{H}.
\end{equation*}
Moreover, if $\Lambda = \mathbb{N}$, then $\mathcal{E}$ is countable and in this case $\big \{ \mathcal{H}; \mathcal{F} = \{\mathcal{H}_\alpha\}_{\alpha \in \Lambda}; \mathcal{D} \big \}$ is called a \textit{quantized Fr\'echet domain}. For details, see Definition 2.3 of \cite{BGP}.

\begin{example} \label{ex;lhs}
Let $e_n$ be a sequence in $\ell^2(\mathbb{N})$, where the $n^\text{th}$ term of the sequence $e_n$ is 1 and 0 elsewhere. Then $\{e_{n}: n \in \mathbb{N}\}$ is a Hilbert basis of $\ell^2(\mathbb{N}).$ For each $k \in \mathbb{N}$, define the closed (in fact, finite dimensional) subspace  $\mathcal{H}_k := \text{span} \{ e_1, e_2,...,e_k \}$. It follows that $\mathcal{H}_{m} \subseteq \mathcal{H}_{n}$ for $m \leq n$. In other words, the family $\mathcal{F} = \big \{\mathcal{H}_n \big \}_{n \in \mathbb{N}}$ forms a strictly inductive system of Hilbert spaces. The inductive limit $\mathcal{D}$ of the strictly inductive system $\mathcal{F} = \big \{\mathcal{H}_n \big \}_{n \in \mathbb{N}}$ is given by $\mathcal{D} = \varinjlim\limits_{n \in \mathbb{N}} \mathcal{H}_n = \bigcup\limits_{n \in \mathbb{N}} \mathcal{H}_n$, and thus
\begin{align*}
\mathcal{D} = \big \{ x = (x_1, x_2, \cdots, ) \in \ell^2(\mathbb{N}) \; : \; x_n =0 \;\text{for all but finitely many}  \; n \in \mathbb{N} \big \}.
\end{align*}
Here $\mathcal{D}$ is the locally Hilbert space and $\overline{\mathcal{D}} = \ell^2(\mathbb{N})$. In other words, $\big \{ \ell^2(\mathbb{N}), \mathcal{F} = \{\mathcal{H}_n\}_{n \in \mathbb{N}}, \mathcal{D} \big \}$ is a quantized Frechet domain. For a more general construction, one can see Example 2.9 of \cite{BGP}. 
\end{example}

In the remaining of this Section, without mentioning explicitly, we assume that $(\Lambda, \leq)$ is a directed poset. Now, we recall the notion of a locally bounded operator between locally Hilbert spaces. Let $\big \{ \mathcal{H}; \mathcal{F} = \{\mathcal{H}_\alpha\}_{\alpha \in \Lambda}; \mathcal{D} \big \}$  and $\big \{ \mathcal{K}; \mathcal{G} = \{\mathcal{K}_\alpha\}_{\alpha \in \Lambda}; \mathcal{O} \big \}$  be two quantized domains. Consider the families $\{ P_\alpha : \mathcal{H} \rightarrow \mathcal{H}_\alpha \}_{\alpha \in \Lambda}$ and $\{ Q_\alpha : \mathcal{K} \rightarrow \mathcal{K}_\alpha\}_{\alpha \in \Lambda}$ of orthogonal projections. Suppose $\mathcal{L}(\mathcal{D}, \mathcal{O})$ denotes the collection of all densely defined (in $\mathcal{H}$) operators from $\mathcal{D}$ to $\mathcal{O}$, and in particular, $\mathcal{L}(\mathcal{D}, \mathcal{D}) = \mathcal{L}(\mathcal{D})$. Let us consider a subclass of $\mathcal{L}(\mathcal{D}, \mathcal{O})$ defined as   
\begin{equation*}
C_{\mathcal{F}, \mathcal{G}}(\mathcal{D}, \mathcal{O}) := \big \{ T \in \mathcal{L}(\mathcal{D}, \mathcal{O}) \; \; : \; \; TP_\alpha = Q_\alpha T P_\alpha \in \mathcal{B}(\mathcal{H}, \mathcal{K}) \; \; \text{for each} \; \alpha \in \Lambda \big \}.
\end{equation*}
Note that, if $T \in \mathcal{L}(\mathcal{D}, \mathcal{O})$, then $T \in C_{\mathcal{F}, \mathcal{G}}(\mathcal{D}, \mathcal{O})$ if and only if $T(\mathcal{H}_{\alpha}) \subseteq \mathcal{K}_{\alpha}$ and $T\big|_{\mathcal{H}_{\alpha}} \in \mathcal{B}(\mathcal{H}_{\alpha}, \mathcal{K}_\alpha)$ for each $\alpha \in \Lambda$ (refer \cite[Section 2.2]{GIS}). In particular, $C_{\mathcal{F}, \mathcal{F}}(\mathcal{D}, \mathcal{D}) = C_{\mathcal{F}}(\mathcal{D})$ and this forms an algebra of all non commutative continuous functions on the locally Hilbert space $\mathcal{D}$ (see \cite[Equation (3.1)]{AD}).  Consider a subclass of $C_{\mathcal{F}, \mathcal{G}}(\mathcal{D}, \mathcal{O})$ given by
\begin{equation*}
C^\ast_{\mathcal{F}, \mathcal{G}}(\mathcal{D}, \mathcal{O}) := \big \{ T \in C_{\mathcal{F}, \mathcal{G}}(\mathcal{D}, \mathcal{O}) \; \; : \; \; Q_\alpha T \subseteq T P_\alpha \; \; \text{for each} \; \alpha \in \Lambda \big \}.    
\end{equation*}
One may note from \cite[Section 2.1]{BGP} that  an operator $T \in \mathcal{L}(\mathcal{D}, \mathcal{O})$ belongs to $C^\ast_{\mathcal{F}, \mathcal{G}}(\mathcal{D}, \mathcal{O})$ if and only if $T(\mathcal{H}_{\alpha}) \subseteq \mathcal{K}_{\alpha}$,  $T(\mathcal{H}_{\alpha}^{\bot}\cap \mathcal{D}) \subseteq \mathcal{K}_{\alpha}^{\bot}\cap \mathcal{O}$ and  $T\big|_{\mathcal{H}_{\alpha}} \in \mathcal{B}(\mathcal{H}_{\alpha}, \mathcal{K}_\alpha)$ for each $\alpha \in \Lambda$ (also, see \cite[Section 2.2]{GIS}).

\begin{definition}\cite[Section 2.2]{GIS} \label{def;lbo}
Let  $\big \{ \mathcal{H}; \mathcal{F} = \{\mathcal{H}_\alpha\}_{\alpha \in \Lambda}; \mathcal{D} \big \}$  and $\big \{ \mathcal{K}; \mathcal{G} = \{\mathcal{K}_\alpha\}_{\alpha \in \Lambda}; \mathcal{O} \big \}$  be two quantized domains. Then an operator $T \in \mathcal{L}(\mathcal{D}, \mathcal{O})$ is said to be locally bounded if $T \in C^\ast_{\mathcal{F}, \mathcal{G}}(\mathcal{D}, \mathcal{O})$.
\end{definition}

Now, we give an attention to $C^\ast_{\mathcal{F}, \mathcal{F}}(\mathcal{D}, \mathcal{D}) = C^\ast_{\mathcal{F}}(\mathcal{D})$ which can be seen as a (proper) subclass of $\mathcal{L}(\mathcal{D})$. Suppose $T \in C^\ast_\mathcal{F}(\mathcal{D})$, then $T$ has an unbounded dual $T^{\bigstar}$ satisfying  \begin{equation*}
\mathcal{D} \subseteq \text{dom}(T^{\bigstar}), \;  \;  \; T^{\bigstar}(\mathcal{D}) \subseteq \mathcal{D}.
\end{equation*}
If we define $T^\ast := T^\bigstar\big|_{\mathcal{D}}$, then it follows that the correspondence $T \mapsto T^\ast$ is an involution on $C^\ast_\mathcal{F}(\mathcal{D})$. Conversely, if $T \in \mathcal{L}(\mathcal{D})$  with
\begin{equation} \label{eqn; locally bounded operator on D}
T(\mathcal{H}_\alpha) \subseteq \mathcal{H}_\alpha; \; \; \; T^\ast(\mathcal{H}_\alpha) \subseteq \mathcal{H}_\alpha; \; \; \; T\big|_{\mathcal{H}_\alpha}, T^\ast\big|_{\mathcal{H}_\alpha} \in \mathcal{B}(\mathcal{H}_\alpha) \; \; \text{for each} \; \; \alpha \in \Lambda,
\end{equation}
then $T \in C^\ast_\mathcal{F}(\mathcal{D})$ \big(see \cite[Proposition 3.1]{AD} \big). As a result $C^\ast_\mathcal{F}(\mathcal{D})$ is unital $\ast$-algebra. Further, 

It is worth to point out that an operator $T \in C^\ast_\mathcal{F}(\mathcal{D})$ when considered as an unbounded operator in $\mathcal{H}$ with the domain $\mathcal{D}$, it is not necessarily closed (see Example \ref{ex;lbo} below). However, every operator in $C^\ast_\mathcal{F}(\mathcal{D})$ is closable. This is because the unbounded dual $T^\bigstar$ is densely defined in $\mathcal{H}$, whenever $T \in C^\ast_\mathcal{F}(\mathcal{D})$.

\begin{example} \label{ex;lbo}
Let $\big \{ \ell^2(\mathbb{N}), \mathcal{F} = \{\mathcal{H}_n\}_{n \in \mathbb{N}}, \mathcal{D} \big \}$ be the quantized Fr\'echet domain as given in Example \ref{ex;lhs}. Now, consider the linear operator $S$ in $\ell^2(\mathbb{N})$ given by the matrix 
\begin{align*}
S= \begin{bmatrix}
1      & 0      & 0      & \cdots & 0      & \cdots \\
0      & 2      & 0      & \cdots & 0      & \cdots \\
0      & 0      & 3      & \cdots & 0      & \cdots \\
\vdots & \vdots & \vdots & \ddots & \vdots & \vdots\\
0      & 0      & 0      & \cdots & m      &  \cdots\\
0      & 0      & 0      & \cdots & 0      & \ddots
\end{bmatrix}.
\end{align*}
with the domain 
\begin{align*}
\text{dom}(S) = \Big \{ \big \{x_n \big\}_{n \in \mathbb{N}} \in \ell^2(\mathbb{N}) \;\; \big| \;\; \sum^\infty_{n =1} n^2 | x_n |^2 < \infty \Big \}.
\end{align*}
Since the locally Hilbert space $\mathcal{D}$ consists of sequences with finite support, this gives the inclusion $\mathcal{D} \subset \text{dom}(S)$. Here the inclusion is proper because the sequence $\big \{\frac{1}{n^2} \big\}_{n \in \mathbb{N}} \in \text{dom}(S) \setminus \mathcal{D}$. Since for each $n \in \mathbb{N}$, we have $Se_n = n e_n$, we obtain $S(\mathcal{H}_n) \subseteq \mathcal{H}_n$ and $S(\mathcal{H}_{n}^{\bot}\cap \mathcal{D}) \subseteq \mathcal{H}_{n}^{\bot}\cap \mathcal{D}$. Moreover, $\big \| S\big|_{\mathcal{H}_n} \big \| = n$ for all $n \in \mathbb{N}$, that is, $S\big|_{\mathcal{H}_n} \in \mathcal{B}(\mathcal{H}_n)$. It follows from Equation \eqref{eqn; locally bounded operator on D} that $T := S\big|_{\mathcal{D}} \in C^\ast_\mathcal{F}(\mathcal{D})$. Equivalently, $T$ is a locally bounded operator on $\mathcal{D}$. 

On the other hand, $T$ can be seen as a densely defined unbounded operator in $\ell^2(\mathbb{N})$ with the domain $\mathcal{D}$. 
However, $T$ is not a closed operator. It can be seen by considering the sequence $\{ Y_n \}_{n \in \mathbb{N}} \subseteq \mathcal{D}$, where
\begin{equation*}
Y_n = \Big \{1, \frac{1}{2^3}, \frac{1}{3^3}, \cdots, \frac{1}{n^3}, 0, 0, \cdots \Big\}, \; \; \; \text{for each} \; \; n \in \mathbb{N}.
\end{equation*}
Then 
\begin{equation*}
T(Y_n) = \big \{1, \frac{1}{2^2}, \frac{1}{3^2}, \cdots, \frac{1}{n^2}, 0, 0, \cdots \big\}, \; \; \; \text{for each} \; \; n \in \mathbb{N}.
\end{equation*}
Now let us consider $Y := \big \{1, \frac{1}{2^3}, \frac{1}{3^3}, \cdots, \frac{1}{n^3}, \frac{1}{{(n+1)}^3}, \cdots \big\}_{n \in \mathbb{N}}$ and $Z := \big \{1, \frac{1}{2^2}, \frac{1}{3^2}, \cdots, \frac{1}{n^2}, \frac{1}{{(n+1)}^2}, \cdots \big\}_{n \in \mathbb{N}}$  in $\ell^2(\mathbb{N})$. Then note that 
\begin{equation*}
\big \|  Y_n - Y \big \|_{\ell^2(\mathbb{N})} \rightarrow 0 \; \; \; \text{and} \; \; \; \big \|  T(Y_n) - Z \big \|_{\ell^2(\mathbb{N})} \rightarrow 0 \; \; \text{as} \; \; n \rightarrow \infty.
\end{equation*}
Since $Y \notin \mathcal{D}$, we conclude that $T$ is not a closed operator.
\end{example}

\begin{remark}\cite[Subection 1.4]{AG} \label{rem;lbo}
Let  $\big \{ \mathcal{H}; \mathcal{F} = \{\mathcal{H}_\alpha\}_{\alpha \in \Lambda}; \mathcal{D} \big \}$  and $\big \{ \mathcal{K}; \mathcal{G} = \{\mathcal{K}_\alpha\}_{\alpha \in \Lambda}; \mathcal{O} \big \}$  be two quantized domains. Consider $T \in C^*_{\mathcal{F}, \mathcal{G}}(\mathcal{D}, \mathcal{O})$, and for each $\alpha \in \Lambda$ take 
$T_\alpha := T\big|_{\mathcal{H}_{\alpha}}$. Then $T_\alpha = Q_{\mathcal{K}_\alpha}T\big|_{\mathcal{H}_{\alpha}}$, where $Q_{\mathcal{K}_\alpha}$ denotes the orthogonal projection of $\mathcal{O}$ to $\mathcal{K}_\alpha$ (see \cite[Lemma 3.1]{AG3} for the existence of such projection). For a fixed $\alpha \in \Lambda$, we denote the inclusion maps by the notations $J_{\mathcal{D}, \alpha} : \mathcal{H}_\alpha \rightarrow \mathcal{D}$ and $J_{\mathcal{O}, \alpha} : \mathcal{K}_\alpha \rightarrow \mathcal{O}$. Then the collection $\{ T_\alpha \}_{\alpha \in \Lambda}$ satisfies the following properties:
\begin{enumerate}
\item for each $\alpha \in \Lambda$, $T_\alpha \in \mathcal{B}(\mathcal{H}_\alpha, \mathcal{K}_\alpha)$ and $T J_{\mathcal{D}, \alpha} = J_{\mathcal{O}, \alpha} T_\alpha$;
\item for $\alpha, \beta \in \Lambda$ with $\alpha \leq \beta$, we get $T^*_\beta\big|_{\mathcal{K}_{\alpha}} = T^*_\alpha$.
\end{enumerate}
\end{remark}
In view of Remark \ref{rem;lbo} and following the notations of \cite{AG}, we say that every $T \in C^{\ast}_{\mathcal{F}}(\mathcal{D})$ can be seen as a projective (or inverse) limit of the net $\{T_{\alpha}\}_{\alpha \in \Lambda}$ of bounded operators. That is, 
\begin{equation} \label{eq; inverese limit of bounded operators}
    T = \varprojlim\limits_{\alpha \in \Lambda} T_{\alpha}.
\end{equation}

Now we are in a position to discuss the notion of locally $C^{\ast}$-algebra. For a comprehensive study of such algebras and local completely positive maps, one can see \cite{BGP, AD, AG, AG2, AI, MJ1, MJ2, MJ3, NCP} and references therein. 
\subsection{Locally $C^{\ast}$-algebra} Let $\mathcal{A}$ be a unital $\ast$-algebra. A seminorm $p$ on $\mathcal{A}$ is said to be a $C^{\ast}$-seminorm, if 
\begin{equation*}
p(a^\ast) = p(a), \; \; \; \; p(a^*a) = p(a)^2, \; \; \; \; p(ab) \leq p(a) p(b),
\end{equation*}
\noindent for all $a,b \in \mathcal{A}$. It is important to note that in \cite{ZS}, the author proved that the condition $p(a^*a) = p(a)^2$ implies $p(ab) \leq p(a) p(b)$ for all $a,b \in \mathcal{A}$. Let $\mathcal{P} := \{ p_\alpha \;  : \;   \alpha \in \Lambda \}$ be a family of $C^{\ast}$-seminorms defined on a the $\ast$-algebra $\mathcal{A}$. Then $\mathcal{P}$ is called a upward filtered family if for each $a \in \mathcal{A}$, we have  $p_\alpha(a) \leq p_\beta(a)$, whenever $\alpha \leq \beta$.

\begin{definition} \cite{AI} \label{def;lca}
A $*$-algebra $\mathcal{A}$ which is complete with respect to the locally convex topology generated by an upward filtered family $\{ p_\alpha \;  : \;  \alpha \in \Lambda \}$ of C*-seminorms is called a locally $C^{\ast}$-algebra. Further, $\mathcal{A}$ is called a unital $*$-algebra, if $\mathcal{A}$ has unit.
\end{definition}

It is well known that every locally $C^{\ast}$-algebra can be realized as the projective limit (or inverse limit) of a projective system (or inverse system) of $C^{\ast}$-algebras. The construction of such projective system is given in \cite{BGP, AG, MJ1, NCP}. However, we recall a few points here: Let $\mathcal{A}$ be a locally $C^{\ast}$-algebra. Then for each $\alpha \in \Lambda$, take $\mathcal{I}_\alpha := \{ a \in \mathcal{A} \; : \; p_\alpha(a) = 0 \}$, which is a two-sided closed ideal in $\mathcal{A}$, then $\mathcal{A}_\alpha : = {\mathcal{A}}\!/\!{\mathcal{I}_\alpha}$ is a $C^{\ast}$-algebra with respect to the $C^{\ast}$-norm induced by $p_\alpha$ (refer \cite{CA, NCP} for more details). Whenever $\alpha \leq \beta$, since $P_\alpha(a) \leq p_\beta(a)$ for each $a \in \mathcal{A}$, there is a $C^\ast$-homomorphism (surjective) $\pi_{\alpha, \beta} : \mathcal{A}_\beta \rightarrow \mathcal{A}_\alpha$ given by $\pi_{\alpha, \beta}(a + \mathcal{I}_\beta) = a + \mathcal{I}_\alpha$. This shows that $\left ( \{ \mathcal{A}_\alpha \}_{\alpha \in \Lambda},  \{ \pi_{\alpha, \beta} \}_{\alpha \leq \beta}\right )$ forms a projective system of $C^{\ast}$-algebras. Then we consider the projective limit of the projective system $\left ( \{ \mathcal{A}_\alpha \}_{\alpha \in \Lambda},  \{ \pi_{\alpha, \beta} \}_{\alpha \leq \beta} \right )$, which is given by \big (refer Subsection 1.1 of \cite{AG} \big)
\begin{align}\label{eq;pl}
\varprojlim\limits_{\alpha \in \Lambda} \mathcal{A}_\alpha := \left \{ \{x_\alpha\}_{\alpha \in \Lambda} \in \prod_{\alpha \in \Lambda} \mathcal{A}_\alpha ~~ : ~~ \pi_{\alpha, \beta}(x_\beta) = x_\alpha, ~~ \text{whenever} ~~ \alpha \leq \beta \right \},
\end{align}
which is equipped with the weakest locally convex topology that makes each linear map $\pi_\beta : \varprojlim\limits_{\alpha \in \Lambda} \mathcal{A}_\alpha \rightarrow \mathcal{A}_\beta$ defined by $\pi_\beta(\{x_\alpha\}_{\alpha \in \Lambda}) := x_\beta$ is a continuous $\ast$-homomorphism. This topology is known as the projective limit topology. Since $\mathcal{A}_\alpha$ is complete for each $\alpha \in \Lambda$, the projective limit $\varprojlim\limits_{\alpha \in \Lambda} \mathcal{A}_\alpha$ is complete with respect to the projective limit topology \big (see Subsection 1.1 of \cite{AG} \big ). Moreover, the pair $\big(\varprojlim\limits_{\alpha \in \Lambda} \mathcal{A}_\alpha, \{ \pi_\alpha \}_{\alpha \in \Lambda} \big)$ is compatible with the projective system $\left ( \{ \mathcal{A}_\alpha \}_{\alpha \in \Lambda},  \{ \pi_{\alpha, \beta} \}_{\alpha \leq \beta} \right )$, that is, $\pi_{\alpha, \beta} \circ \pi_\beta = \pi_\alpha$, whenever  $\alpha \leq \beta$. 

Let $\mathcal{W}$ be a locally convex $\ast$-algebra along with the net of continuous $\ast$-homomorphisms given by $\big \{\psi_\alpha : \mathcal{W} \rightarrow \mathcal{A}_\alpha \big\}_{\alpha \in \Lambda}$. If the pair $(\mathcal{W}, \{ \psi_\alpha \}_{\alpha \in \Lambda})$ is compatible with the projective system $\left ( \{ \mathcal{A}_\alpha \}_{\alpha \in \Lambda},  \{ \pi_{\alpha, \beta} \}_{\alpha \leq \beta} \right )$, that is, 
\begin{equation}\label{eq;ps}
\pi_{\alpha, \beta} \circ \psi_\beta = \psi_\alpha    \; \text{whenever} \; \alpha \leq \beta, 
\end{equation}
then there always exists a unique continuous linear map $\Psi : \mathcal{W} \rightarrow \varprojlim\limits_{\alpha \in \Lambda} \mathcal{A}_\alpha$ satisfying $\psi_\alpha = \pi_\alpha \circ \Psi$ for each $\alpha \in \Lambda$.  In this sense, the projective limit $\big (\varprojlim\limits_{\alpha \in \Lambda} \mathcal{A}_\alpha, \{ \pi_\alpha \}_{\alpha \in \Lambda} \big)$ of the projective system $\left ( \{ \mathcal{A}_\alpha \}_{\alpha \in \Lambda},  \{ \pi_{\alpha, \beta} \}_{\alpha \leq \beta} \right )$ is uniquely determined. A reader is directed to \cite{AG, NCP} for more details.

Now observe that for each $\alpha \in \Lambda$, there is a canonical projection map $\phi_\alpha : \mathcal{A} \rightarrow \mathcal{A}_\alpha$  satisfying $\pi_{\alpha, \beta} \circ \phi_\beta = \phi_\alpha$, whenever $\alpha \leq \beta$. Equivalently, we get that the pair $(\mathcal{A}, \{ \phi_\alpha \}_{\alpha \in \Lambda})$ is compatible with the projective system  $\left ( \{ \mathcal{A}_\alpha \}_{\alpha \in \Lambda},  \{ \pi_{\alpha, \beta} \}_{\alpha \leq \beta} \right )$ \big (see Subsection 1.1 of \cite{AG} \big ). Moreover, 
the projective limit $\big(\varprojlim\limits_{\alpha \in \Lambda} \mathcal{A}_\alpha, \{ \phi_\alpha \}_{\alpha \in \Lambda} \big)$ of the projective system $\left ( \{ \mathcal{A}_\alpha \}_{\alpha \in \Lambda},  \{ \pi_{\alpha, \beta} \}_{\alpha \leq \beta} \right )$ can be identified with the locally $C^\ast$-algebra $\mathcal{A}$ and the identification is given by the continuous  bijective $\ast$-homomorphism $\phi : \mathcal{A} \rightarrow \varprojlim\limits_{\alpha \in \Lambda} \mathcal{A}_\alpha$ defined by $a \mapsto \{a + \mathcal{I}_\alpha\}_{\alpha \in \Lambda}$. 
By referring to the identification given by the map $\phi$, we write $\mathcal{A} = \varprojlim\limits_{\alpha \in \Lambda} \mathcal{A}_\alpha$. One may refer to \cite{AG, NCP} and references therein for more details.


\begin{remark} \label{rem;pl,sh} \cite{KS}
Let $\left ( \{ \mathcal{B}_\alpha \}_{\alpha \in \Lambda},  \{ \psi_{\alpha, \beta} \}_{\alpha \leq \beta} \right )$ 
be a projective system of $C^\ast$-algebras. Then the projective limit $\varprojlim\limits_{\alpha \in \Lambda} \mathcal{B}_\alpha$ is a locally $C^\ast$-algebra  and it is uniquely determined by defining as in Equation \eqref{eq;pl} 
\begin{equation*}
\varprojlim\limits_{\alpha \in \Lambda} \mathcal{B}_\alpha := \left \{ \{x_\alpha\}_{\alpha \in \Lambda} \in \prod_{\alpha \in \Lambda} \mathcal{B}_\alpha ~~ : ~~ \psi_{\alpha, \beta}(x_\beta) = x_\alpha, ~~ \text{whenever} ~~ \alpha \leq \beta \right \},
\end{equation*}
along with the continuous $\ast$-homomorphisms $\pi_\beta : \varprojlim\limits_{\alpha \in \Lambda} \mathcal{B}_\alpha \rightarrow \mathcal{B}_\beta$ defined as $\pi_\beta \big  (\{ x_\alpha \}_{\alpha \in \Lambda} \big  ) := x_\beta$ for every $\beta \in \Lambda$. 
\end{remark}

\begin{example} \cite[(1) of Example 1.4]{AG}\label{ex;lca}
Let  $\big \{ \mathcal{H}, \mathcal{F} = \{\mathcal{H}_\alpha\}_{\alpha \in \Lambda}, \mathcal{D} \big \}$ be a quantized domain. For each fixed $\beta \in \Lambda$, define the branch of $\Lambda$ determined by $\beta$ as $\Lambda_\beta = \{ \alpha \in \Lambda \;  :  \; \alpha \leq \beta \}$ . Then, with the induced order from $(\Lambda, \leq)$ the set $(\Lambda_\beta, \leq)$ is a directed poset (see \cite[Section 1.4]{AG}). For each $\beta \in \Lambda$ consider the quantized domain $\big \{ \mathcal{H}_\beta, \mathcal{F}_\beta = \{\mathcal{H}_\alpha\}_{\alpha \in \Lambda_\beta}, \mathcal{H}_\beta \big \}$. Then the $\ast$-algebra $C^*_{\mathcal{F}_\beta}(\mathcal{H}_\beta)$ is a $C^\ast$-subalgebra of $\mathcal{B}(\mathcal{H}_\beta)$ (see \cite[(1) of Example 1.4]{AG}). Now for $\alpha \leq \beta$, define a $\ast$-homomorphism 
\begin{align*}
\phi_{\alpha, \beta} : C^*_{\mathcal{F}_\beta}(\mathcal{H}_\beta) \rightarrow C^*_{\mathcal{F}_\alpha}(\mathcal{H}_\alpha) \; \; \; \text{as} \; \;  \; \phi_{\alpha, \beta}(S) = S \big|_{\mathcal{H}_\alpha},
\end{align*}
for every $S \in C^*_{\mathcal{F}_\beta}(\mathcal{H}_\beta)$. Then $\left ( \big \{ C^*_{\mathcal{F}_\alpha}(\mathcal{H}_\alpha) \big \}_{\alpha \in \Lambda},  \big \{\phi_{\alpha, \beta} \big \}_{\alpha \leq \beta} \right )$ forms a projective system of $C^{\ast}$-algebras. Now let $T \in C^*_\mathcal{F}(\mathcal{D})$. By Equation \eqref{eq; inverese limit of bounded operators} and Definition \ref{def;lbo}, we have $T = \varprojlim\limits_{\alpha \in \Lambda} T_\alpha$, where  $T_\alpha \in \mathcal{B}(\mathcal{H}_\alpha)$ and in particular $T_\alpha \in C^*_{\mathcal{F}_\alpha}(\mathcal{H}_\alpha)$. For each $\alpha \in  \Lambda$, define a surjective $\ast$-homomorphism  $\phi_\alpha : C^*_\mathcal{F}(\mathcal{D}) \rightarrow C^*_{\mathcal{F}_\alpha}(\mathcal{H}_\alpha)$ by  $\phi_\alpha(T) := T_\alpha$. Then we have $\phi_{\alpha, \beta} \circ \phi_\beta = \phi_\alpha$, whenever $\alpha \leq \beta$. This shows that the pair $ \big ( C^*_\mathcal{F}(\mathcal{D}), \{ \phi_\alpha \}_{\alpha \in \Lambda} \big )$ is compatible with the projective system given by $\left ( \big \{ C^*_{\mathcal{F}_\alpha}(\mathcal{H}_\alpha) \big \}_{\alpha \in \Lambda},  \big \{\phi_{\alpha, \beta} \big \}_{\alpha \leq \beta} \right )$ \big (see Subsection 1.1 of \cite{AG} \big ). Further, the pair $\big (C^*_\mathcal{F}(\mathcal{D}), \{ \phi_\alpha \}_{\alpha \in \Lambda} \big )$ satisfies the universal property. This means let $\mathcal{W}$ be a locally convex $\ast$-algebra along with a net of continuous $\ast$-homomorphisms $\big \{\psi_\alpha : \mathcal{W} \rightarrow C^*_{\mathcal{F}_\alpha}(\mathcal{H}_\alpha) \big\}_{\alpha \in \Lambda}$. If the pair $ \big (\mathcal{W}, \{ \psi_\alpha \}_{\alpha \in \Lambda} \big )$ is compatible with the projective system $\left ( \big \{ C^*_{\mathcal{F}_\alpha}(\mathcal{H}_\alpha) \big \}_{\alpha \in \Lambda},  \big \{\phi_{\alpha, \beta} \big \}_{\alpha \leq \beta} \right )$, that is, 
\begin{equation}\label{eq;ps}
\phi_{\alpha, \beta} \circ \psi_\beta = \psi_\alpha    \; \text{whenever} \; \alpha \leq \beta, 
\end{equation}
then there always exists a unique continuous linear map $\Psi : \mathcal{W} \rightarrow C^*_\mathcal{F}(\mathcal{D})$ satisfying $\psi_\alpha = \phi_\alpha \circ \Psi$ for each $\alpha \in \Lambda$. Thus by the uniquenss of the projective limit we obtain 
\begin{equation*}
C^*_\mathcal{F}(\mathcal{D}) = \varprojlim\limits_{\alpha \in \Lambda} C^*_{\mathcal{F}_\alpha}(\mathcal{H}_\alpha).
\end{equation*}
This shows that $C^*_\mathcal{F}(\mathcal{D})$ is a locally $C^\ast$-algebra. Moreover, by referring to Equation (1.33) of \cite{AG}, for each $\alpha \in \Lambda$, one can see that a seminorm defined by $p_\alpha : C^*_\mathcal{F}(\mathcal{D}) \rightarrow \mathbb{R}$ as $p_\alpha(T) := \| T_\alpha \|$ is a $C^\ast$-seminorm. Then the family $\{ p_\alpha \}_{\alpha \in \Lambda}$ of $C^\ast$-seminorms induces a locally convex topology on $C^*_\mathcal{F}(\mathcal{D})$. 
\end{example}

\subsection{Locally von Neumann algebra} In this subsection, we recall some results from the notion of locally von Neumann algebra. First, we recall a definition of locally von Neumann algebra given in \cite{MF}.

\begin{definition}\cite[Definition 1.1]{MF} \label{def; lva 1}
An algebra $\mathcal{M}$ is said to be a locally von Neumann algebra if $\mathcal{M}$ is the projective limit of some projective system of von Neumann algebras. 
\end{definition}

Next, we follow an approach to define a locally von Neumann algebra given in \cite{MJ1}. Let  $\big \{ \mathcal{H}; \mathcal{F} = \{\mathcal{H}_\alpha\}_{\alpha \in \Lambda}; \mathcal{D} \big \}$  be a quantized domain. If $u, v \in \mathcal{D}$, then $u, v \in \mathcal{H}_\gamma$ for some $\gamma \in \Lambda$ and we define
\begin{equation*}
q_{u}(T) := \| Tu \|_{\mathcal{H}_\gamma} \; \; \text{and} \; \;q_{u,v}(T) := \left | \langle u, Tv \rangle_{\mathcal{H}_\gamma} \right | \; \; \text{for all} \; T \in C^*_\mathcal{F}(\mathcal{D}).
\end{equation*}
Thus $q_u$ and $q_{u, v}$ are $C^\ast$-seminorms. Then
\begin{enumerate}
\item[(a)] (SOT) {\it strong operator topology} on $C^*_\mathcal{F}(\mathcal{D})$ is the locally convex topology generated by the family $\{q_u ~~ : ~~ u \in \mathcal{D} \}$ of $C^\ast$-seminorms;
\item[(b)] (WOT) {\it weak operator topology} on $C^*_\mathcal{F}(\mathcal{D})$ is the locally convex topology generated by the family $\{q_{u,v} ~~ : ~~ u, v \in \mathcal{D} \}$ of $C^\ast$-seminorms.
\end{enumerate}
For a detailed introduction to locally von Neumann algebras, a reader is directed to \cite{MF, MJ1, MJ2}. In \cite{MJ1}, the author showed that, in the case of $\ast$-subslegrbras of $C^*_\mathcal{F}(\mathcal{D})$), Definition \ref{def; lva 1} coincides with the following definition of locally von Neumann algebras.

\begin{definition}\cite[Definition 3.7]{MJ1} \label{def; lva 2}
Let  $\big \{ \mathcal{H}; \mathcal{F} = \{\mathcal{H}_\alpha\}_{\alpha \in \Lambda}; \mathcal{D} \big \}$  be a quantized domain. Then a locally von Neumann algebra is a strongly closed unital locally $C^{\ast}$-algebra contained in $C^*_\mathcal{F}(\mathcal{D})$ (with the same unit as in $C^*_\mathcal{F}(\mathcal{D})$).
\end{definition}

Let $\mathcal{M} \subseteq C^*_\mathcal{F}(\mathcal{D})$. Consider the set $\mathcal{M}^\prime := \{ T \in C^*_\mathcal{F}(\mathcal{D}) ~~:~~ TS = ST ~~ \text{for all} ~~ S \in \mathcal{M} \}$ which is called as the commutant of $\mathcal{M}$ (see \cite{MJ1}). We denote $(\mathcal{M}^\prime)^\prime$ by the notation $\mathcal{M}^{\prime \prime}$. The following theorem proved in \cite{MJ1} is the double commutant theorem in the setting of locally von Neumann algebra. 

\begin{theorem}\cite[Theorem 3.6]{MJ1}
Let  $\big \{ \mathcal{H}; \mathcal{F} = \{\mathcal{H}_\alpha\}_{\alpha \in \Lambda}; \mathcal{D} \big \}$  be a quantized domain, and let $\mathcal{M} \subseteq C^*_\mathcal{F}(\mathcal{D})$ be a locally $C^{\ast}$-algebra containing the identity operator on $\mathcal{D}$. Then the following statements are equivalent:
\begin{enumerate}
\item $ \mathcal{M} = \mathcal{M}^{\prime \prime}$;
\item $\mathcal{M}$ is weakly closed;
\item $\mathcal{M}$ is strongly closed.
\end{enumerate}
\end{theorem}

\section{Direct integral of locally Hilbert spaces} \label{sec; Direct integrals}
Motivated from the theory of direct integral of Hilbert spaces, we propose an approach to define the notion of direct integral of locally Hilbert spaces over a standard measure space.

Throughout this section, we assume that $(\Lambda, \leq)$ is a directed poset, and  $(X, \mu)$ is a measure space as described in Note \ref{note; measure space}. In certain examples, we consider specific cases of $(\Lambda, \leq)$ and  $(X, \mu)$, which we state explicitly.  For each $p \in X$,  assign a quantized domain $\big \{ \mathcal{H}_p; \mathcal{E}_p = \{\mathcal{H}_{\alpha, p}\}_{\alpha \in \Lambda}; \mathcal{D}_p \big \}$. For each fixed $\alpha \in \Lambda$, we consider a family $\{\mathcal{H}_{\alpha, p}\}_{p \in X}$ of Hilbert spaces. The direct integral of $\{ \mathcal{H}_{\alpha, p} \}_{p \in X}$ over $(X, \mu)$ is denoted by $\diX \mathcal{H}_{\alpha, p} \, \dmu$ and it is a Hilbert space (see Definition \ref{def;dihs}). In this way, we obtain a family  $\mathcal{E} := \left \{ \int^\oplus_{X} \mathcal{H}_{\alpha, p} \, \dmu \right \}$ of Hilbert spaces. Whenever $\alpha \leq \beta$, by using the fact $\mathcal{H}_{\alpha, p} \subseteq \mathcal{H}_{\beta, p}$ for each fixed $p \in X$ and following Definition \ref{def;dihs}, we see that 
\begin{equation*}
\diX \mathcal{H}_{\alpha, p} \, \dmu \subseteq \diX \mathcal{H}_{\beta, p} \, \dmu.
\end{equation*}
Thus the family $\mathcal{E} = \left \{\int^\oplus_{X} \mathcal{H}_{\alpha, p} \, \dmu \right\}_{\alpha \in \Lambda}$ forms a strictly inductive system of Hilbert spaces.

\begin{definition} \label{Defn: directint_loc}
For each $p \in X$,  we associate a quantized domain $\big \{ \mathcal{H}_p; \mathcal{E}_p = \{\mathcal{H}_{\alpha, p}\}_{\alpha \in \Lambda}; \mathcal{D}_p \big \}$. Then the direct integral of locally Hilbert spaces $\{ \mathcal{D}_p \}_{p \in X}$ is denoted by $\displaystyle \dilX \mathcal{D}_p \, \dmu$ and defined as
\begin{equation} \label{eqn: directint_loc}
\displaystyle \dilX \mathcal{D}_p \, \dmu :=  \varinjlim\limits_{\alpha \in \Lambda}  \int^\oplus_{X} \mathcal{H}_{\alpha, p} \, \dmu = \bigcup\limits_{\alpha \in \Lambda} \left (\int^\oplus_{X} \mathcal{H}_{\alpha, p} \, \dmu \right ) 
\end{equation}
\end{definition}

From Definition \ref{Defn: directint_loc} it is clear that the notion of the direct integral of locally Hilbert spaces $\{ \mathcal{D}_p \}_{p \in X}$ given by $\displaystyle \dilX \mathcal{D}_p \, \dmu$ is the limit of a strictly inductive system $\mathcal{E} = \left \{ \int^\oplus_{X} \mathcal{H}_{\alpha, p} \, \dmu \right \}_{\alpha \in \Lambda}$ of Hilbert spaces. Thus $\displaystyle \dilX \mathcal{D}_p \, \dmu$ is again a locally Hilbert space (see Definition \ref{def;lhs}).

\begin{remark} \label{rem; interchange of limit and integral}
For a family of quantized domain $\big \{ \mathcal{H}_p; \mathcal{E}_p = \{\mathcal{H}_{\alpha, p}\}_{\alpha \in \Lambda}; \mathcal{D}_p \big \}_{p \in X}$ by following Equation \eqref{eqn: directint_loc}, we get 
\begin{equation*}
\dilX \varinjlim\limits_{\alpha \in \Lambda} \mathcal{H}_{\alpha, p} \, \dmu   = \displaystyle \dilX \mathcal{D}_p \, \dmu =  \varinjlim\limits_{\alpha \in \Lambda}  \int^\oplus_{X} \mathcal{H}_{\alpha, p} \, \dmu.
\end{equation*}
This shows that the notion of direct integral of locally Hilbert spaces as defined in Definition \ref{Defn: directint_loc} gives an affirmative answer to the fundamental question of whether the inductive limit and the direct integral can be interchanged.
\end{remark}

\begin{proposition}  \label{prop; QD}
Let $\big \{ \mathcal{H}_p; \mathcal{E}_p = \{\mathcal{H}_{\alpha, p}\}_{\alpha \in \Lambda}; \mathcal{D}_p \big \}_{p \in X}$ be a family of quantized domains. Then 
\begin{equation*}
\left \{ \int^\oplus_{X} \mathcal{H}_p \, \dmu; \mathcal{E} = \left \{\int^\oplus_{X} \mathcal{H}_{\alpha, p} \, \dmu \right\}_{\alpha \in \Lambda}; \dilX \mathcal{D}_p \, \dmu \right \}
\end{equation*}
is a quantized domain.   
\end{proposition}
\begin{proof}
We have shown that  $\mathcal{E} = \left \{\int^\oplus_{X} \mathcal{H}_{\alpha, p} \, \dmu \right\}_{\alpha \in \Lambda}$ forms a strictly inductive system of Hilbert spaces and it follows from Equation \eqref{eqn: directint_loc} that $\displaystyle \dilX \mathcal{D}_p \, \dmu  = \bigcup\limits_{\alpha \in \Lambda} \left (\int^\oplus_{X} \mathcal{H}_{\alpha, p} \, \dmu \right )$. Note that the direct integral of $\{ \mathcal{H}_p \}_{p \in X}$  over  $(X, \mu)$ which is  denoted by $\diX \mathcal{H}_p \, \dmu$ is a Hilbert space (see Definition \ref{def;dihs}).  Thus to complete the proof, it remains to show that $\int^\oplus_{X} \mathcal{H}_p \, \dmu = \overline{\displaystyle\dilX \mathcal{D}_p \, \dmu}$. Let $x = \int^\oplus_{X} x(p) \, \dmu$ be an arbitrarily chosen element in $\int^\oplus_{X} \mathcal{H}_p \, \dmu$. Then from Definition \ref{def;dihs}, we have $x(p) \in \mathcal{H}_p$ for all $p \in X$. For each fixed $\alpha \in \Lambda$ and $p \in X$, we know that $\mathcal{H}_{\alpha, p}$ is a closed subspace of the Hilbert space $\mathcal{H}_p$ and so there is a canonical projection $Q_{\alpha, p} : \mathcal{H}_p \rightarrow \mathcal{H}_{\alpha, p}$. Using this, for each $\alpha \in \Lambda$, we define a map $u_\alpha : X \rightarrow  \bigcup\limits_{p \in X} \mathcal{H}_{\alpha, p}$ as
\begin{equation*}
u_\alpha(p) := 
Q_{\alpha, p} x(p), \; \;  \text{for all} \;\; p \in X.
\end{equation*} 
Then one may observe that for any $y = \int^\oplus_{X} y(p) \, \dmu \in \int^\oplus_{X} \mathcal{H}_p \, \dmu$, we have
\begin{align*}
\int_X \big | \la  u_\alpha(p), y(p) \ra \big | \, \dmu   &\leq \int_X \big \| u_\alpha(p) \big  \| \big \|y(p) \big \|  \, \dmu \\
& \leq \int_X \big \| x(p) \big \| \big \|y(p) \big \|  \, \dmu < \infty,
\end{align*}
where the last inequality holds valid as both $x, y \in \diX \mathcal{H}_p \, \dmu$. Thus by using Condition (2) of Definition \ref{def;dihs}, we get that $u_\alpha \in \int^\oplus_{X} \mathcal{H}_{\alpha, p} \, \dmu$, and hence $u_\alpha \in \displaystyle \dilX \mathcal{D}_p \, \dmu$ for all $\alpha \in \Lambda$. In summary, we obtain a net $\{ u_\alpha \}_{\alpha \in \Lambda}$ in $\displaystyle \dilX \mathcal{D}_p \, \dmu$.

Our aim is to show that the element $x \in \int^\oplus_{X} \mathcal{H}_p \, \dmu$ can be approximated by the net $\{ u_\alpha \}_{\alpha \in \Lambda}$ in $\diX \mathcal{H}_p \, \dmu$. To see this, let us define functions $f_\alpha : X \rightarrow \mathbb{R}$ and $g : X \rightarrow \mathbb{R}$ by  
\begin{equation*}
f_\alpha(p) := \big \| u_\alpha(p) - x(p) \big \|_{\mathcal{H}_p}^2, \; \; \;  \; \; \; \; \;  \text{and} \; \; \;  \; \; \; \; \; g(p) := \big \| x(p) \big \|_{\mathcal{H}_p}^2 \; \; \; \text{for all} \; \; \; p \in X.
\end{equation*}
Then for all $p \in X$, we get $f_\alpha(p) \rightarrow 0$, and we observe
\begin{align*}
f_\alpha(p) &=  \big \| u_\alpha(p) - x(p) \big \|_{\mathcal{H}_p}^2 \\
&= \big \| Q_{\alpha, p} x(p) - x(p) \big \|_{\mathcal{H}_p}^2 \\
&= \big \| Q_{\alpha, p} x(p) \big \|_{\mathcal{H}_p}^2 + \big \| x(p) \big \|_{\mathcal{H}_p}^2 - \la  Q_{\alpha, p} x(p), x(p)  \ra_{\mathcal{H}_p} -  \la x(p), Q_{\alpha, p} x(p)  \ra_{\mathcal{H}_p} \\
&= \big \| Q_{\alpha, p} x(p) \big \|_{\mathcal{H}_p}^2 + \big \| x(p) \big \|_{\mathcal{H}_p}^2 - 2 \big\| Q_{\alpha, p} x(p) \big \|_{\mathcal{H}_p}^2  \\
&= \big \| x(p) \big \|_{\mathcal{H}_p}^2 - \big \| Q_{\alpha, p} x(p) \big \|_{\mathcal{H}_p}^2 \\
&\leq  \big \| x(p) \big \|_{\mathcal{H}_p}^2 = g(p) \; \; \; \; \text{for all} \; \; p \in X.
\end{align*}
Since $x \in \int^\oplus_{X} \mathcal{H}_p \, \dmu$, by following condition (1) of Definition \ref{def;dihs}, we know that $\int_{X} g(p) \, \dmu < \infty$. Thus by applying the dominated convergence theorem, we obtain 
\begin{equation*}
\big \| u_\alpha - x \big \|_{\diX \mathcal{H}_p \, \dmu}^2 = \int_{X} f_\alpha(p) \, \dmu \rightarrow  0.
\end{equation*}
As $x  \in \int^\oplus_{X} \mathcal{H}_p \, \dmu$ was arbitrarily chosen, we obtain $$\int^\oplus_{X} \mathcal{H}_p \, \dmu = \overline{\displaystyle\dilX \mathcal{D}_p \, \dmu}.$$ This completes the proof.
\end{proof}

\begin{remark} \label{rem; examples of dilhs particular cases}
We wish to point out two particular cases of direct integral of locally Hilbert spaces.
\begin{enumerate} 
\item \label{rem; examples of dilhs particular cases_lhs is hs} In Definition \ref{Defn: directint_loc}, if $\mathcal{D}_p = \mathcal{H}_p$ for a.e. $p \in X$, that is, $\mathcal{D}_p$ is a Hilbert space for a.e. $p \in X$, then the notion of the direct integral of locally Hilbert spaces introduced in Definition \ref{Defn: directint_loc} coincides with the classical notion of direct integral Hilbert spaces (given as in Definition \ref{def;dihs}). In this case, we obtain $\int^\oplus_{X} \mathcal{H}_p \, \dmu = \displaystyle \dilX \mathcal{D}_p \, \dmu $.
\item \label{rem; examples of dilhs particular cases_direct sum} Consider a measure space $\big ( X = \{ 1, 2, 3, ..., n \}, \mu \big )$, where $\mu$ is the counting measure on $X$. Let $\big \{ \mathcal{H}_i; \mathcal{E}_i = \{\mathcal{H}_{\alpha, i}\}_{\alpha \in \Lambda}; \mathcal{D}_i \big \}_{i \in X}$ be a family of quantized domains. The direct sum of locally Hilbert spaces $\big \{ \mathcal{D}_i \big \}_{i \in X}$  is defined as $\bigoplus\limits_{i =1}^n \mathcal{D}_i := \bigcup\limits_{\alpha \in \Lambda}  \left (\bigoplus\limits_{i =1}^n \mathcal{H}_{\alpha, i} \right )$ (see \cite[Section 3.2]{AD}). Then by following the classical result, $\displaystyle \diX \mathcal{H}_{\alpha, i} \, \mathrm{d} \mu(i)  = \bigoplus\limits_{i =1}^n \mathcal{H}_{\alpha, i}$ (refer \cite[Examples 14.1.4 (b)]{KR2}) for each $\alpha \in \Lambda$, one may observe that
\begin{equation*}
\displaystyle \dilX \mathcal{D}_i \, \mathrm{d} \mu(i) = \bigcup\limits_{\alpha \in \Lambda} \left ( \displaystyle \diX \mathcal{H}_{\alpha, i} \, \mathrm{d} \mu(i) \right ) = \bigcup\limits_{\alpha \in \Lambda}  \left (\bigoplus\limits_{i =1}^n \mathcal{H}_{\alpha, i} \right ) = \bigoplus\limits_{i =1}^n \mathcal{D}_i.
\end{equation*}
\end{enumerate}
\end{remark}

In the example below, we discuss about the infinite direct sum of locally Hilbert spaces which is a particular case of direct integral of locally Hilbert spaces.

\begin{example} \label{ex;dilhs1 with same standard spaces}
Consider a measure space $\big ( \mathbb{N}, \mu \big )$, where $\mu$ is the counting measure on $\mathbb{N}$. Let $\big \{ \mathcal{H}_n; \mathcal{E}_n = \{\mathcal{H}_{\alpha, n}\}_{\alpha \in \Lambda}; \mathcal{D}_n \big \}_{n \in \mathbb{N}}$ be a family of quantized domains. By extending the notion of the (finite) direct sum of locally Hilbert spaces to the infinite direct sum of $\big \{ \mathcal{D}_n \big \}_{n \in \mathbb{N}}$,  we get
\begin{equation*}
\bigoplus\limits_{n =1}^\infty \mathcal{D}_n := \big \{ x = (x_1, x_2, ...) \; : \; x_n \in \mathcal{D}_n  \; \; \text{for all} \; n \in \mathbb{N} \; \; \text{and} \; \; \text{supp}(x) < \infty \; \big \}.
\end{equation*}
Then clearly, $\bigoplus\limits_{n =1}^\infty \mathcal{D}_n := \bigcup\limits_{\alpha \in \Lambda}  \left (\bigoplus\limits_{n =1}^\infty \mathcal{H}_{\alpha, n} \right )$. On the other hand, for each fixed $\alpha \in \Lambda$, by following \cite[Examples 14.1.4 (b)]{KR2} we obtain, $\int^\oplus_{\mathbb{N}} \mathcal{H}_{\alpha, n} \, \mathrm{d} \mu(n)  = \bigoplus\limits_{n =1}^\infty \mathcal{H}_{\alpha, n}$. Therefore, 
\begin{equation*}
\displaystyle \int^{\oplus_\text{loc}}_\mathbb{N} \mathcal{D}_n \; \mathrm{d} \mu(n) = \bigcup\limits_{\alpha \in \Lambda} \left ( \displaystyle \int^{\oplus}_\mathbb{N} \mathcal{H}_{\alpha, n} \, \mathrm{d} \mu(n) \right ) = \bigcup\limits_{\alpha \in \Lambda}  \left (\bigoplus\limits_{n =1}^\infty \mathcal{H}_{\alpha, n} \right ) = \bigoplus\limits_{n =1}^\infty \mathcal{D}_n.
\end{equation*}
\end{example}

\begin{example} \label{ex;dilhs2}
Consider $\Lambda = \mathbb{N}$ and the quantized Frechet domain $\big \{ \ell^2(\mathbb{N}), \mathcal{F} = \{\mathcal{H}_n\}_{n \in \mathbb{N}}, \mathcal{D} \big \}$ as given in Example \ref{ex;lhs}. Next, consider a measure space $(\mathbb{R}, \mu)$. For each $p \in \mathbb{R}$, assign a  quantized domain 
\begin{equation*}
\left \{ \mathcal{H}_p = \ell^2(\mathbb{N}); \mathcal{E}_p = \{ \mathcal{H}_{n, p} = \mathcal{H}_n\}_{n \in \mathbb{N}}; \mathcal{D}_p = \mathcal{D} \right \}.    
\end{equation*}
Firstly, note that $\mathcal{H}_p = \ell^2(\mathbb{N})$ for each $p \in \mathbb{R}$, by following the Corollary in page no. 176 of \cite{DixV} we get a unitary operator 
$U : \int^{\oplus_\text{loc}}_\mathbb{R} \mathcal{H}_p \, \dmu \rightarrow \text{L}^2 \big (\mathbb{R}, B(\mathbb{R}), \mu \big) \otimes \ell^2(\mathbb{N})$. Moreover, for each fixed $n \in \mathbb{N}$ 
\begin{equation*}
\int^{\oplus}_\mathbb{R} \mathcal{H}_{n, p} \, \dmu = \int^{\oplus}_\mathbb{R} \mathcal{H}_{n} \, \dmu = U^\ast \left ( \text{L}^2 \big (\mathbb{R}, B(\mathbb{R}), \mu \big) \otimes \mathcal{H}_n \right )
\end{equation*}
Since each $\mathcal{D}_p = \mathcal{D}$, from Definition \ref{Defn: directint_loc}, we have
\begin{align*}
\displaystyle \int^{\oplus_\text{loc}}_\mathbb{R} \mathcal{D} \, \dmu = \displaystyle \int^{\oplus_\text{loc}}_\mathbb{R} \mathcal{D}_p \, \dmu 
&= \varinjlim\limits_{n \in \mathbb{N}}  \int^{\oplus}_\mathbb{R} \mathcal{H}_{n, p} \, \dmu \\
&= \varinjlim\limits_{n \in \mathbb{N}} U^\ast \left ( \text{L}^2 \big (\mathbb{R}, B(\mathbb{R}), \mu \big) \otimes \mathcal{H}_n \right ) \\
&= U^\ast \left ( \varinjlim\limits_{n \in \mathbb{N}}   \text{L}^2 \big (\mathbb{R}, B(\mathbb{R}), \mu \big) \otimes \mathcal{H}_n \right ) \\
& = U^\ast \left  (\text{L}^2 \big (\mathbb{R}, B(\mathbb{R}), \mu \big) \otimes_{\text{loc}} \mathcal{D} \right )
\end{align*}
where the last equality follows from \cite[Equation (1.23)]{AG}. Therefore, the locally Hilbert spaces $\displaystyle \int^{\oplus_\text{loc}}_\mathbb{R} \mathcal{D} \, \dmu$ and $\text{L}^2 \big (\mathbb{R}, B(\mathbb{R}), \mu \big) \otimes_{\text{loc}} \mathcal{D}$ can be identified through the unitary operator $U$. 
\end{example}

\section{Decomposable and Diagonalizable Locally Bounded Operators} \label{sec; Decomposable and Diagonalizable Locally Bounded Operators}
Throughout this section, we assume that $(\Lambda, \leq)$ is a directed poset, and  $(X, \mu)$ is a measure space as described in Note \ref{note; measure space}. In certain results, we make use of specific cases of $(\Lambda, \leq)$ and  $(X, \mu)$, which we state explicitly. Consider a family $\big \{ \mathcal{H}_p; \mathcal{E}_p = \{\mathcal{H}_{\alpha, p}\}_{\alpha \in \Lambda}; \mathcal{D}_p \big \}_{p \in X}$  of quantized domains. Then by following Proposition \ref{prop; QD}, we get a quantized domain 
\begin{equation*}
\left \{ \int^\oplus_{X} \mathcal{H}_p \, \dmu; \mathcal{E} = \left \{\int^\oplus_{X} \mathcal{H}_{\alpha, p} \, \dmu \right\}_{\alpha \in \Lambda}; \dilX \mathcal{D}_p \, \dmu \right \}.
\end{equation*}
The collection of all locally bounded operators on the locally Hilbert space $\displaystyle \dilX \mathcal{D}_p \, \dmu$ is denoted by  $C^\ast_{\mathcal{E}}\left (\displaystyle \dilX \mathcal{D}_p \, \dmu \right)$ (see Definition \ref{def;lbo}). Now, we turn our attention towards sub-classes of  $C^\ast_{\mathcal{E}}\left (\displaystyle \dilX \mathcal{D}_p \, \dmu \right)$ that adhere to the notion of direct integrals. We introduce these classes motivated from the classical setup of direct integral of Hilbert spaces, particularly from the notion of decomposable and diagonalizable bounded operators (see Definition \ref{def;Debo}).

\begin{definition} \label{def;DecDiag(lbo)}
Let $\big \{ \mathcal{H}_p; \mathcal{E}_p = \{\mathcal{H}_{\alpha, p}\}_{\alpha \in \Lambda}; \mathcal{D}_p \big \}_{p \in X}$ be a family of quantized domains.
 Then a locally bounded operator $T \in C^\ast_{\mathcal{E}}\left (\displaystyle \dilX \mathcal{D}_p \, \dmu \right)$ is said to be:
\begin{enumerate} 
\item \label{def;Dec(lbo)} \textbf{decomposable}, if there exists a family $ \big \{ T_p \in C^\ast_{\mathcal{E}_p}\left ( \mathcal{D}_p \right) \big \}_{p \in X}$ of locally bounded operators such that for any $u \in \displaystyle \dilX \mathcal{D}_p \, \dmu$, we have
\begin{align*}
(Tu)(p) = T_pu(p) \; \; \; \; \text{for a.e.} \; \; p \in X.
\end{align*}
In this case, we denote the operator $T$ by the notation $\displaystyle \dilX T_p \, \dmu$ and so
\begin{equation*}
\left (\dilX T_p \, \dmu \right ) \left (\dilX u(p) \, \dmu \right )= \dilX T_pu(p) \, \dmu;
\end{equation*}
\item \label{def;Diag(lbo)} \textbf{diagonalizable}, if $T$ is decomposable and there exists a measurable function $f : X \rightarrow \mathbb{C}$ such that for any $u \in \displaystyle \dilX \mathcal{D}_p \, \dmu$, we have 
\begin{equation*}
(Tu)(p) = f(p)u(p) \; \; \; \; \text{for a.e.} \; \; p \in X.
\end{equation*}
In this situation, we get $T = \displaystyle \dilX T_p \, \dmu = \displaystyle \dilX f(p) \cdot \mathrm{Id}_{\mathcal{D}_p} \, \dmu$.
\end{enumerate}  
\end{definition}

\noindent
We denote the collection of all decomposable locally bounded operators and the collection of all diagonalizable locally bounded operators on $\displaystyle \dilX \mathcal{D}_p \, \dmu$ by $C^\ast_{\mathcal{E}, \text{DEC}}\left (\displaystyle \dilX \mathcal{D}_p \, \dmu \right)$ and $C^\ast_{\mathcal{E}, \text{DIAG}}\left (\displaystyle \dilX \mathcal{D}_p \, \dmu \right)$ respectively. From Definition \ref{def;DecDiag(lbo)}, one may observe that
\begin{equation} \label{eqn; containment}
C^\ast_{\mathcal{E}, \text{DIAG}}\left (\displaystyle \dilX \mathcal{D}_p \, \dmu \right) \subseteq C^\ast_{\mathcal{E}, \text{DEC}}\left (\displaystyle \dilX \mathcal{D}_p \, \dmu \right) \subseteq C^\ast_{\mathcal{E}}\left (\displaystyle \dilX \mathcal{D}_p \, \dmu \right).
\end{equation}

\subsection{Observations} \label{obs; Observations} Consider a family $\big \{ \mathcal{H}_p; \mathcal{E}_p = \{\mathcal{H}_{\alpha, p}\}_{\alpha \in \Lambda}; \mathcal{D}_p \big \}_{p \in X}$ of quantized domains. Then by following Proposition \ref{prop; QD}, we obtain a quantized domain given by 
\begin{equation*}
\left \{ \int^\oplus_{X} \mathcal{H}_p \, \dmu; \mathcal{E} = \left \{\int^\oplus_{X} \mathcal{H}_{\alpha, p} \, \dmu \right\}_{\alpha \in \Lambda}; \displaystyle \dilX \mathcal{D}_p \, \dmu \right \}.
\end{equation*}
The following are the key observations that are useful in understanding the notion of  decomposable and diagonalizable locally bounded operators on $\displaystyle \dilX \mathcal{D}_p \, \dmu$.

\begin{enumerate} 
\item \label{obs;M DEC and M DIAG are star algebras} 
The class of decomposable locally bounded operators on $\displaystyle \dilX \mathcal{D}_p \, \dmu$ is denoted by $C^\ast_{\mathcal{E}, \text{DEC}}\left (\displaystyle \dilX \mathcal{D}_p \, \dmu \right)$, which is a locally convex $\ast$-subalgebra of $C^\ast_\mathcal{E}\left(\displaystyle \dilX \mathcal{D}_p \, \dmu \right)$ with respect to the following operations: 
\begin{multicols}{2}
\begin{enumerate}
\item $T + S = \displaystyle \dilX T_p + S_p \, \dmu$\\
\item $\lambda \cdot T = \displaystyle \dilX \lambda \cdot T_p \, \dmu$
\item $T \cdot S = \displaystyle \dilX T_p \cdot S_p \, \dmu$\\
\item $T^\ast = \displaystyle \dilX T^\ast_p \, \dmu$,
\end{enumerate}
\end{multicols}
\noindent
for every $T = \displaystyle \dilX T_p \, \dmu, \; S = \displaystyle \dilX S_p \, \dmu$ in $C^\ast_\mathcal{E}\left(\displaystyle \dilX \mathcal{D}_p \, \dmu \right)$ and $\lambda \in \mathbb{C}$. Similarly, $C^\ast_{\mathcal{E}, \text{DIAG}}\left (\displaystyle \dilX \mathcal{D}_p \, \dmu \right)$  forms a locally convex $\ast$-subalgebra of  $C^\ast_\mathcal{E}\left(\displaystyle \dilX \mathcal{D}_p \, \dmu \right)$.

\item \label{obs; description of V alpha T V alpha star}
Suppose $T \in C^\ast_{\mathcal{E}, \text{DEC}}\left (\displaystyle \dilX \mathcal{D}_p \, \dmu \right)$, then $T$ is a locally bounded operator and there is a family $\big \{T_p \in C^\ast_{\mathcal{E}_p}\left (\mathcal{D}_p \right) \big \}_{p \in X}$ such that $T = \displaystyle \dilX T_p \, \dmu$. Since $\displaystyle \dilX \mathcal{D}_p \, \dmu  = \bigcup\limits_{\alpha \in \Lambda} \left (\int^\oplus_{X} \mathcal{H}_{\alpha, p} \, \dmu \right ) $. (see Definition \ref{Defn: directint_loc}), for each $\alpha \in \Lambda$, we obtain a decomposable bounded linear operator on $\int^\oplus_{X} \mathcal{H}_{\alpha, p} \, \dmu$ as
\begin{equation} \label{eq;restriction of DecLBO}
T_\alpha := T\big|_{\int^\oplus_{X} \mathcal{H}_{\alpha, p} \, \dmu} = \int^\oplus_{X} T_{p} \big|_{\mathcal{H}_{\alpha, p}} \,  \mathrm{d} \mu(p).
\end{equation}
Moreover, from Equation \eqref{eq;norm of T}, we have 
\begin{equation} \label{eqn; norm of T alpha when T is decomposable}
\| T_\alpha \| = \text{ess} \sup\limits_{p \in X} \big \{ \big \| T_{p}\big|_{\mathcal{H}_{\alpha, p}} \big\| \big \} < \infty.
\end{equation}

\item \label{obs; description of T alpha for diagonalizable} In particular, if $T \in C^\ast_{\mathcal{E}, \text{DIAG}}\left (\displaystyle \dilX \mathcal{D}_p \, \dmu \right)$, then 
$T = \displaystyle \dilX f(p) \cdot \mathrm{Id}_{\mathcal{D}_{p}} \, \dmu$, for some measurable function $f : X \rightarrow \mathbb{C}$ and  for each $\alpha \in \Lambda$ the operator 
\begin{equation} \label{eq;restriction of DiagLBO}
T_\alpha := T\big|_{\int^\oplus_{X} \mathcal{H}_{\alpha, p} \, \dmu} = \int^\oplus_{X} f(p) \cdot \mathrm{Id}_{\mathcal{H}_{\alpha, p}} \,  \mathrm{d} \mu (p)
\end{equation}
is bounded. Moreover, from Equations \eqref{eqn; containment} and \eqref{eqn; norm of T alpha when T is decomposable} we have
\begin{equation}
\| T_\alpha \| = \text{ess} \sup\limits_{p \in X} \big \{ \big \| T_{p}\big|_{\mathcal{H}_{\alpha, p}} \big\| \big \}  =  \text{ess} \sup\limits_{p \in X} \big \{ \big \| f(p) \big\| \big \} < \infty,
\end{equation}
for every $\alpha \in \Lambda$. Hence $f \in \text{L}^\infty \big (X, \mu \big )$.
Since $\| T_\alpha \|  =  \text{ess} \sup\limits_{p \in X} \big \{ \big \| f(p) \big\| \big \}$ for each $\alpha \in \Lambda$, we get $\text{ess} \sup\limits_{\alpha \in \Lambda} \big \{ \big \| T_\alpha \big\| \big \} < \infty$ and so by \cite[Example 1.14]{AG} $T$ is a bounded operator on the locally Hilbert space $\displaystyle \dilX \mathcal{D}_p \, \dmu$. Thus $T$ can be canonically extended to the Hilbert space completion of $\displaystyle \dilX \mathcal{D}_p \, \dmu$ which is given by $\diX \mathcal{H}_p \, \dmu$ (see Proposition \ref{prop; QD}). This shows that every operator in $C^\ast_{\mathcal{E}, \text{DIAG}}\left (\displaystyle \dilX \mathcal{D}_p \, \dmu \right)$ is in fact bounded and it corresponds to a function in the space $\text{L}^\infty \big (X, \mu \big )$. 
\end{enumerate}

\begin{theorem}
Let $\big \{ \mathcal{H}_p; \mathcal{E}_p = \{\mathcal{H}_{\alpha, p}\}_{\alpha \in \Lambda}; \mathcal{D}_p \big \}$ be a quantized domain for each $p \in X$. Then \; $C^\ast_{\mathcal{E}, \text{DIAG}} \left(\displaystyle \dilX \mathcal{D}_p \, \dmu \right)$ can be canonically embeded in $\mathcal{B} \left(\diX \mathcal{H}_p \, \dmu \right)$ as an abelian von Neumann algebra.
\end{theorem}
\begin{proof}
Let $T$ be an element of $C^\ast_{\mathcal{E}, \text{DIAG}} \left(\displaystyle \dilX \mathcal{D}_p \, \dmu \right)$. Then from \ref{obs; description of T alpha for diagonalizable} of Observation \ref{obs; Observations} we know that $T = \displaystyle \dilX f(p) \cdot \mathrm{Id}_{\mathcal{D}_{p}} \, \dmu$, for some  $f \in \text{L}^\infty \big (X, \mu \big )$ and in fact $T$ is a bounded operator. Since the locally Hilbert space $\displaystyle \dilX \mathcal{D}_p \, \dmu$ is dense in the Hilbert space $\diX \mathcal{H}_p \, \dmu$ (see Proposition \ref{prop; QD}), $T$ has a unique extension to a bounded operator (which we denote by $\widetilde{T}$) on  $\diX \mathcal{H}_p \, \dmu$. Using this, we define a map $\Phi : C^\ast_{\mathcal{E}, \text{DIAG}} \left(\displaystyle \dilX \mathcal{D}_p \, \dmu \right) \rightarrow  \mathcal{B} \left(\diX \mathcal{H}_p \, \dmu \right)$ by
\begin{equation*}
\Phi \left  (T = \displaystyle \dilX f(p) \cdot \mathrm{Id}_{\mathcal{D}_{p}} \, \dmu \right ) :=  \widetilde{T} = \displaystyle \diX f(p) \cdot \mathrm{Id}_{\mathcal{H}_{p}} \, \dmu ,
\end{equation*}
for every $ T \in C^\ast_{\mathcal{E}, \text{DIAG}} \left(\displaystyle \dilX \mathcal{D}_p \, \dmu \right)$. From the definition of the map $\Phi$, it is clear that $\Phi$ is an injective $\ast$-homomorphism. Moreover it follows from \ref{def;Diagbo} of Definition \ref{def;Debo} and \ref{obs; description of T alpha for diagonalizable} of Observation \ref{obs; Observations} that $\Phi \left ( C^\ast_{\mathcal{E}, \text{DIAG}} \left(\displaystyle \dilX \mathcal{D}_p \, \dmu \right) \right )$  coincides with the abelian von Neumann algebra of all diagonalizable bounded operators on $\diX \mathcal{H}_p \, \dmu$. This proves the result.
\end{proof}

Next, in some particular cases, we prove that the collection of all decomposable locally bounded operators on the direct integral of locally Hilbert spaces is a locally von Neumann algebra. 

\begin{theorem} \label{thm;DEC and DIAG LvNA}
Let $\big \{ \mathcal{H}_p; \mathcal{E}_p = \{\mathcal{H}_{\alpha, p}\}_{\alpha \in \Lambda}; \mathcal{D}_p \big \}$ be a quantized domain for each $p \in X$. Suppose either $\Lambda$ is a countable set or $\mu$ is a countable measure on $X$, then $C^\ast_{\mathcal{E}, \text{DEC}} \left(\displaystyle \dilX \mathcal{D}_p \, \dmu \right)$ is a locally von Neumann algebra.
\end{theorem}
\begin{proof}
For each $\beta \in \Lambda$, consider the set $\Lambda_\beta = \{ \alpha \in \Lambda \;  :  \; \alpha \leq \beta \}$, the branch of $\Lambda$ determined by $\beta$. Then with the induced order $\leq$ from $(\Lambda, \leq)$ the set $(\Lambda_\beta, \leq)$ is a directed poset (see \cite[Section 1.4]{AG}). For each fixed $\beta \in \Lambda$, consider a strictly inductive system $\mathcal{E}_\beta = \left \{\int^\oplus_{X} \mathcal{H}_{\alpha, p} \, \dmu \right\}_{\alpha \in \Lambda_\beta}$ of Hilbert subspaces of the Hilbert space $ \int^{\oplus}_{X} \mathcal{H}_{\beta, p} \, \dmu$.
Now recall that $C^\ast_{\mathcal{E}_\beta} \left( \int^{\oplus}_{X} \mathcal{H}_{\beta, p} \, \dmu \right)$ consists of all bounded operators on the Hilbert space $\int^{\oplus}_{X} \mathcal{H}_{\beta, p} \, \dmu$ for which the Hilbert space $\int^{\oplus}_{X} \mathcal{H}_{\alpha, p} \, \dmu$ is a reducing subspace, whenever $\alpha \leq \beta$. In particular, $C^\ast_{\mathcal{E}_\beta, \text{DEC}} \left( \int^{\oplus}_{X} \mathcal{H}_{\beta, p} \, \dmu \right)$ consists of all decomposable bounded operators on the Hilbert space $\int^{\oplus}_{X} \mathcal{H}_{\beta, p} \, \dmu$ for which the Hilbert space $\int^{\oplus}_{X} \mathcal{H}_{\alpha, p} \, \dmu$ is a reducing subspace, whenever $\alpha \leq \beta$.

\noindent
\textbf{Claim 1:} For each  $\beta\in \Lambda$, \; $C^\ast_{\mathcal{E}_\beta, \text{DEC}} \left( \int^{\oplus}_{X} \mathcal{H}_{\beta, p} \, \dmu \right)$ is a von Neumann algebra in $\mathcal{B}\left( \int^{\oplus}_{X} \mathcal{H}_{\beta, p} \, \dmu \right)$. 

From \ref{obs;M DEC and M DIAG are star algebras} of Observations \ref{obs; Observations}, we know that this space is a $\ast$-algebra. Thus to prove our claim, it suffices to show that it is closed in the strong operator topology.  To see this, let $\left \{ T_n \right \}_{n \in \mathbb{N}}$ be a sequence in $C^\ast_{\mathcal{E}_\beta, \text{DEC}} \left( \int^{\oplus}_{X} \mathcal{H}_{\beta, p} \, \dmu \right)$ such that $T_n \rightarrow T$ in the strong operator topology on $\mathcal{B}\left( \int^{\oplus}_{X} \mathcal{H}_{\beta, p} \, \dmu \right)$. Here $T$ is decomposable since  $C^\ast_{\mathcal{E}_\beta, \text{DEC}} \left(\int^{\oplus}_{X} \mathcal{H}_{\beta, p} \, \dmu \right)$ is contained in the von Neumann algebra of all decomposable operators on the Hilbert space $\int^{\oplus}_{X} \mathcal{H}_{\beta, p} \, \dmu$, and thus $T =  \int^\oplus_{X} T_{p} \, \dmu$. Now we show that for every $\alpha \leq \beta$ the subspace $\int^{\oplus}_{X} \mathcal{H}_{\alpha, p} \, \dmu$ is  reducing for $T$. Let  $x\in \int^{\oplus}_{X} \mathcal{H}_{\beta, p} \, \dmu$. Then $x \in \int^{\oplus}_{X} \mathcal{H}_{\alpha, p} \, \dmu$ for some $\alpha \leq \beta$. Since $\int^{\oplus}_{X} \mathcal{H}_{\alpha, p} \, \dmu$ is a reducing subspace for each $T_n$ we obtain $T_n(x) \in \int^{\oplus}_{X} \mathcal{H}_{\alpha, p} \, \dmu$. Then we have $Tx \in \int^{\oplus}_{X} \mathcal{H}_{\alpha, p} \, \dmu$ since $\big \|  T_nx - Tx \big \|_{\int^{\oplus}_{X} \mathcal{H}_{\alpha, p} \, \dmu} \rightarrow 0$, as $n \to \infty$. Therefore, $\int^{\oplus}_{X} \mathcal{H}_{\alpha, p} \, \dmu$ is  reducing for $T$ for every $\alpha \leq \beta$ and hence $T \in C^\ast_{\mathcal{E}_\beta, \text{DEC}} \left( \int^{\oplus}_{X} \mathcal{H}_{\beta, p} \, \dmu \right)$, which proves our Claim 1.

\noindent \textbf{Claim 2:} $C^\ast_{\mathcal{E}, \text{DEC}} \left(\displaystyle \dilX \mathcal{D}_p \, \dmu \right) = \varprojlim\limits_{\alpha \in \Lambda} C^\ast_{\mathcal{E}_\alpha, \text{DEC}} \left(\int^{\oplus}_{X} \mathcal{H}_{\alpha, p} \, \dmu \right)$.

Note that if $T \in C^\ast_{\mathcal{E}_\beta, \text{DEC}} \left( \int^{\oplus}_{X} \mathcal{H}_{\beta, p} \, \dmu \right)$, then $\int^{\oplus}_{X} \mathcal{H}_{\alpha, p} \, \dmu$ is a reducing subspace for $T$, whenever $\alpha \leq \beta$. Using this, for $\alpha \leq \beta$, we define a map $\phi_{\alpha, \beta} : C^\ast_{\mathcal{E}_\beta, \text{DEC}} \left( \int^{\oplus}_{X} \mathcal{H}_{\beta, p} \, \dmu \right) \rightarrow C^\ast_{\mathcal{E}_\alpha, \text{DEC}} \left(  \int^{\oplus}_{X} \mathcal{H}_{\alpha, p} \, \dmu \right)$ by
\begin{equation} \label{eq; phi alpha beta}
\phi_{\alpha, \beta} \left ( T \right ) := T\big|_{\int^{\oplus}_{X} \mathcal{H}_{\alpha, p} \, \dmu}, 
\end{equation}
for every $T \in C^\ast_{\mathcal{E}_\beta, \text{DEC}} \left( \int^{\oplus}_{X} \mathcal{H}_{\beta, p} \, \dmu \right)$. Clearly $\big \{ \phi_{\alpha, \beta} \big \}_{\alpha \leq \beta}$ is a family of  normal $\ast$-homomorphisms of von Neumann algebras with the property that $\phi_{\alpha, \alpha} = \mathrm{Id}_{\int^{\oplus}_{X} \mathcal{H}_{\alpha, p} \, \dmu}$ and $\phi_{\alpha, \beta} \circ \phi_{\beta, \gamma} = \phi_{\alpha, \gamma}$, whenever $\alpha \leq \beta \leq \gamma$. This shows that $\left (  \left \{ C^\ast_{\mathcal{E}_\alpha, \text{DEC}} \left(\int^{\oplus}_{X} \mathcal{H}_{\alpha, p} \, \dmu \right) \right \}_{\alpha \in \Lambda}, \{ \phi_{\alpha, \beta} \}_{\alpha \leq \beta} \right )$ is a projective system of von Neumann algebras. We know from (\ref{obs;M DEC and M DIAG are star algebras}) of Observations \ref{obs; Observations} that $C^\ast_{\mathcal{E}, \text{DEC}} \left(\displaystyle \dilX \mathcal{D}_p \, \dmu \right)$ is a locally convex $\ast$-algebra with respect to the $C^\ast$-seminorm given for each $\alpha \in \Lambda$, as $T \mapsto \left  \| T\big|_{\int^{\oplus}_{X} \mathcal{H}_{\alpha, p} \, \dmu}  \right \|$ (see Example \ref{ex;lca}). Thus for each $\alpha \in \Lambda$, the map $\phi_\alpha : C^\ast_{\mathcal{E}, \text{DEC}} \left(\displaystyle \dilX \mathcal{D}_p \, \dmu \right) \rightarrow   C^\ast_{\mathcal{E}_\alpha, \text{DEC}} \left(\int^{\oplus}_{X} \mathcal{H}_{\alpha, p} \, \mathrm{d} \mu_\alpha \right)$ defined as 
\begin{equation} \label{eq; phi alpha}
\phi_\alpha \left ( T \right ) := T\big|_{\int^{\oplus}_{X} \mathcal{H}_{\alpha, p} \, \dmu},
\end{equation}
is a continuous $\ast$-homomorphism satisfying $\phi_{\alpha, \beta} \circ \phi_\beta = \phi_\alpha$, whenever $\alpha \leq \beta$. This shows that the pair $ \left (C^\ast_{\mathcal{E}, \text{DEC}} \left(\displaystyle \dilX \mathcal{D}_p \, \dmu \right) , \{ \phi_\alpha \}_{\alpha \in \Lambda} \right )$ is compatible with the projective system given by $\left (  \left \{ C^\ast_{\mathcal{E}_\alpha, \text{DEC}} \left(\int^{\oplus}_{X} \mathcal{H}_{\alpha, p} \, \dmu \right) \right \}_{\alpha \in \Lambda}, \{ \phi_{\alpha, \beta} \}_{\alpha \leq \beta} \right )$  of von Neumann algebras  \big ( see Subsection 1.1 of \cite{AG} \big ). To prove  that $ \left (C^\ast_{\mathcal{E}, \text{DEC}} \left(\displaystyle \dilX \mathcal{D}_p \, \dmu \right) , \{ \phi_\alpha \}_{\alpha \in \Lambda} \right )$ is the projective limit (in the category of topological algebras) of  $\left (  \left \{ C^\ast_{\mathcal{E}_\alpha, \text{DEC}} \left( \int^{\oplus}_{X} \mathcal{H}_{\alpha, p} \, \dmu \right) \right \}_{\alpha \in \Lambda}, \{ \phi_{\alpha, \beta} \}_{\alpha \leq \beta} \right )$,  we consider a locally convex $\ast$-algebra $\mathcal{W}$ with a family $\{ \psi_\alpha \}_{\alpha \in \Lambda}$, where $\psi_\alpha : \mathcal{W} \rightarrow  C^\ast_{\mathcal{E}_\alpha, \text{DEC}} \left( \int^{\oplus}_{X} \mathcal{H}_{\alpha, p} \, \mathrm{d} \mu_\alpha \right)$ is a continuous $\ast$-homomorphism for each $\alpha \in \Lambda$.  Let the pair $ \big (\mathcal{W} , \{ \psi_\alpha \}_{\alpha \in \Lambda} \big )$ be compatible with the projective system given by $\left (  \left \{ C^\ast_{\mathcal{E}_\alpha, \text{DEC}} \left( \int^{\oplus}_{X} \mathcal{H}_{\alpha, p} \, \dmu \right) \right \}_{\alpha \in \Lambda}, \{ \phi_{\alpha, \beta} \}_{\alpha \leq \beta} \right )$, that is,
\begin{equation} \label{eqn; compatible for W and phi alpha}
\phi_{\alpha, \beta} \circ \psi_\beta = \psi_\alpha \; \; \text{whenever} \; \; \alpha \leq \beta.
\end{equation}
For each fixed $w \in \mathcal{W}$ and $\alpha \in \Lambda$, since $\psi_\alpha(w)$ is decomposable, we have $\psi_\alpha(w) =  \int^\oplus_{X}  \psi_\alpha(w)_p \, \dmu$ for some family $\left \{ \psi_\alpha(w)_p  \right \}_{p \in X}$ of bounded operators on $\mathcal{H}_{\alpha, p}$ satisfying the property that $\mathcal{H}_{\delta, p}$ is reducing subspace for $\psi_\alpha(w)_p$ for each $\delta \in \Lambda_{\alpha}$. Whenever $\alpha \leq \beta$, for each $w \in \mathcal{W}$, by following Equation \eqref{eqn; compatible for W and phi alpha}, we obtain
\begin{align*}
\int^\oplus_{X}  \psi_\alpha(w)_p \, \dmu = \psi_\alpha(w) &= \phi_{\alpha, \beta} \circ \psi_\beta(w) \\
&= \psi_\beta(w) \big|_{\int^{\oplus}_{X} \mathcal{H}_{\alpha, p} \, \dmu} \\
&= \int^\oplus_{X}  \psi_\beta(w)_p\big|_{\mathcal{H}_{\alpha, p}} \, \dmu.
\end{align*}
Thus, there exists a measurable set $E^w_{\alpha, \beta} \subseteq X$ such that $\mu(E^w_{\alpha, \beta}) = 0$ and $ \psi_\beta(w)_p \big |_{\mathcal{H}_{\alpha, p}} \neq  \psi_\alpha(w)_p$ for every $p \in E^w_{\alpha, \beta}$. Let $E^w_\alpha := \bigcup\limits_{\beta \in \Lambda, \alpha \leq \beta} E^w_{\alpha, \beta}$ and $E^w :=  \bigcup\limits_{\alpha \in \Lambda} E^w_{\alpha}$.

Now assume that $\Lambda$ is countable, then $E^w_\alpha$ is a measurable set for each $\alpha \in \Lambda$. This implies that $E^w$ is also a measurable set, and, moreover, $\mu(E^w) = 0$. On the other hand, if $\mu$ is a counting measure on $X$, then for any $\alpha \leq \beta$, the set $E^w_{\alpha, \beta} = \emptyset$ and thus $E^w = \emptyset$. So in both cases (that is, when $\Lambda$ is countable or $\mu$ is a counting measure on $X$), we consider the set $X \setminus E^w$ and without loss of generality we again denote it by $X$. It follows that $\psi_\beta(w)_p\big|_{\mathcal{H}_{\alpha, p}} = \psi_\alpha(w)_p$ for all $p \in X$, whenever $\alpha \leq \beta$. Using this, for each fixed $w \in \mathcal{W}$ and $p \in X$, define a locally bounded operator $T^w_p \in C^\ast_{\mathcal{E}_{p}} (\mathcal{D}_{p})$ as $T^w_p := \varprojlim\limits_{\alpha \in \Lambda} \psi_\alpha(w)_p$ (see Equation \eqref{eq; inverese limit of bounded operators}). Now we define a map $\Psi : \mathcal{W} \rightarrow C^\ast_{\mathcal{E}, \text{DEC}} \left(\displaystyle \dilX \mathcal{D}_p \, \dmu \right)$ by
\begin{equation*}
w \mapsto \displaystyle \dilX T^w_p  \, \dmu \; \; \; \text{for each} \; \; w \in \mathcal{W}.
\end{equation*}
Next, for $w \in \mathcal{W}$ and $\alpha \in \Lambda$, by following Equation \eqref{eq; phi alpha}, we obtain 
\begin{align*}
\phi_\alpha \circ \Psi(w) = \phi_\alpha \left ( \displaystyle \dilX T^w_p  \, \dmu \right ) &=  \left ( \displaystyle \dilX T^w_p  \, \dmu \right ) \big|_{\int^{\oplus}_{X} \mathcal{H}_{\alpha, p} \, \dmu} \\
&= \int^\oplus_{X} T^w_p \big|_{\mathcal{H}_{\alpha, p}} \,  \dmu \\
&= \int^\oplus_{X} \psi_\alpha(w)_p \,  \dmu = \psi_\alpha(w).
\end{align*}
Since $w \in \mathcal{W}$ and $\alpha \in \Lambda$ were chosen arbitrarily, we get $\phi_\alpha \circ \Psi = \psi_\alpha$ for each $\alpha \in \Lambda$. Now it remains to show that the map $\Psi$ satisfying $\phi_\alpha \circ \Psi = \psi_\alpha$ for each $\alpha \in \Lambda$ is unique. Suppose there exists another map $\hat{\Psi} : \mathcal{W} \rightarrow C^\ast_{\mathcal{E}, \text{DEC}} \left(\displaystyle \dilX \mathcal{D}_p \, \dmu \right)$ such that for all $\alpha \in \Lambda$, we have $\phi_\alpha \circ \hat{\Psi} = \psi_\alpha$. For $w \in \mathcal{W}$ and $\alpha \in \Lambda$,  by following Equation \eqref{eq; phi alpha}, we see that
\begin{align*}
\Psi(w)\big|_{\int^{\oplus}_{X} \mathcal{H}_{\alpha, p} \, \dmu} = \phi_\alpha \circ \Psi(w) =  \psi_\alpha(w) = \phi_\alpha \circ \hat{\Psi}(w) =  \hat{\Psi}(w) \big|_{\int^{\oplus}_{X} \mathcal{H}_{\alpha, p} \, \dmu}.
\end{align*}
As the above equation holds for each $\alpha \in \Lambda$, we get $\Psi(w) = \hat{\Psi}(w)$. Since $w \in \mathcal{W}$ was chosen arbitrarily, we get  that the map $\Psi$ is unique. Therefore, by the uniqueness of the projective limit (see  \cite[Section 1.1]{AG}), we obtain
\begin{align*}
C^\ast_{\mathcal{E}, \text{DEC}} \left(\displaystyle \dilX \mathcal{D}_p \, \dmu \right) = \varprojlim\limits_{\alpha \in \Lambda} C^\ast_{\mathcal{E}_\alpha, \text{DEC}} \left(\int^{\oplus}_{X} \mathcal{H}_{\alpha, p} \, \dmu \right).
\end{align*}
This proves Claim 2. 

Finally, by using Claim 1, Claim 2 and Definition \ref{def; lva 1}, we conclude that $C^\ast_{\mathcal{E}, \text{DEC}} \left(\displaystyle \dilX \mathcal{D}_p \, \dmu \right) $ is a locally von Neumann algebra.
\end{proof}

Let $\big \{ \mathcal{H}; \mathcal{F} = \{\mathcal{H}_{\alpha}\}_{\alpha \in \Lambda}; \mathcal{D} \big \}$  be a quantized domain, and let $\mathcal{M} \subseteq C^*_\mathcal{F}(\mathcal{D})$. We recall that the commutant of $\mathcal{M}$ is denoted by $\mathcal{M}^\prime$ and is defined as 
\begin{equation} \label{eqn; commutant}
\mathcal{M}^\prime := \{ T \in C^*_\mathcal{F}(\mathcal{D}) ~~:~~ TS = ST ~~ \text{for all} ~~ S \in \mathcal{M} \}.    
\end{equation}
Consider a family  $\big \{ \mathcal{H}_p; \mathcal{E}_p = \{\mathcal{H}_{\alpha, p}\}_{\alpha \in \Lambda}; \mathcal{D}_p \big \}_{p \in X}$  of quantized domains. Recall that the collection of all locally bounded operators on $\displaystyle \dilX \mathcal{D}_p \, \dmu$ is denoted by $C^\ast_\mathcal{E} \left ( \displaystyle \dilX \mathcal{D}_p \, \dmu  \right )$, where $\mathcal{E} = \left \{\int^\oplus_{X} \mathcal{H}_{\alpha, p} \, \dmu \right\}_{\alpha \in \Lambda}$. By motivating from  Theorem \ref{thm;DeDibo vNA}, in the remaining of this section, we study the relationship between the subcollections of $C^\ast_\mathcal{E} \left ( \displaystyle \dilX \mathcal{D}_p \, \dmu  \right )$, in particular, \; \; $C^\ast_{\mathcal{E}, \text{DEC}} \left(\displaystyle \dilX \mathcal{D}_p \, \dmu \right)$ and $\left ( C^\ast_{\mathcal{E}, \text{DIAG}} \left(\displaystyle \dilX \mathcal{D}_p \, \dmu \right) \right )^ \prime$.

\begin{remark} \label{rem; M Dec and M Diag commutant}
Let $\big \{ \mathcal{H}_p; \mathcal{E}_p = \{\mathcal{H}_{\alpha, p}\}_{\alpha \in \Lambda}; \mathcal{D}_p \big \}_{p \in X}$  be a family of quantized domains.  Then we get the following containments
\begin{align*}
 C^\ast_{\mathcal{E}, \text{DEC}} \left(\displaystyle \dilX \mathcal{D}_p \, \dmu \right) &\subseteq  \left (C^\ast_{\mathcal{E}, \text{DIAG}} \left(\displaystyle \dilX \mathcal{D}_p \, \dmu \right) \right)^\prime ;  \\
 C^\ast_{\mathcal{E}, \text{DIAG}} \left(\displaystyle \dilX \mathcal{D}_p \, \dmu \right) &\subseteq \left (C^\ast_{\mathcal{E}, \text{DEC}} \left(\displaystyle \dilX \mathcal{D}_p \, \dmu \right) \right)^\prime.
\end{align*}
To see the above containments, let $T \in C^\ast_{\mathcal{E}, \text{DEC}} \left(\displaystyle \dilX \mathcal{D}_p \, \dmu \right)$ and $S \in C^\ast_{\mathcal{E}, \text{DIAG}} \left(\displaystyle \dilX \mathcal{D}_p \, \dmu \right)$. Then by following Definition \ref{def;DecDiag(lbo)} we get a family $ \big \{ T_p \in C^\ast_{\mathcal{E}_p} \left(\mathcal{D}_p \right) \big \}_{p \in X}$ of locally bounded operators and a measurable function $f : X \rightarrow \mathbb{C}$ such that for any $u \in \displaystyle \dilX \mathcal{D}_p \, \dmu$, we have
\begin{align*}
(Tu)(p) = T_pu(p)  \; \; \text{and}  \; \;  (Su)(p) = f(p)u(p) \; \; \text{for a.e.} \; \; p \in X.
\end{align*}
So for a fixed $u = \displaystyle \dilX u(p) \, \dmu \in \displaystyle \dilX \mathcal{D}_p \, \dmu$, we get
\begin{align*}
\big (TS \big ) \left (\dilX u(p) \, \dmu \right ) &= T \left(\dilX f(p)u(p) \, \dmu \right) \\
&= \dilX T_p f(p) u(p) \, \dmu \\
&= \dilX  f(p) T_p u(p) \, \dmu \\
&= S \left (\dilX T_pu(p) \, \dmu \right) \\
&= \big  (ST \big ) \left(\dilX u(p) \, \dmu \right )
\end{align*}
Since $T \in C^\ast_{\mathcal{E}, \text{DEC}} \left(\displaystyle \dilX \mathcal{D}_p \, \dmu \right), \; S \in C^\ast_{\mathcal{E}, \text{DIAG}} \left(\displaystyle \dilX \mathcal{D}_p \, \dmu \right)$ and $u \in \displaystyle \dilX \mathcal{D}_p \, \dmu$ were arbitrarily chosen, we obtain the desired containments.   
\end{remark}

\begin{lemma} \label{lem; V_alphaTV_alpha star is decomposable}
Let $\big \{ \mathcal{H}_p; \mathcal{E}_p = \{\mathcal{H}_{\alpha, p}\}_{\alpha \in \Lambda}; \mathcal{D}_p \big \}_{p \in X}$  be a family of quantized domains, and let $T \in \left (C^\ast_{\mathcal{E}, \text{DIAG}} \left(\displaystyle \dilX \mathcal{D}_p \, \dmu \right) \right)^\prime$. Then for each fixed $\alpha \in \Lambda$, the restriction $T_\alpha := T\big |_{\int^\oplus_{X} \mathcal{H}_{\alpha, p} \, \dmu}$ is a decomposable bounded operator on the Hilbert space $\int^\oplus_{X} \mathcal{H}_{\alpha, p} \, \dmu$.
\end{lemma}
\begin{proof}
Let $T \in \left ( C^\ast_{\mathcal{E}, \text{DIAG}} \left(\displaystyle \dilX \mathcal{D}_p \, \dmu \right) \right )^ \prime$. By following Equation \eqref{eqn; commutant},  $T$ is a locally bounded operator and so, the restriction $T_\alpha = T\big |_{\int^\oplus_{X} \mathcal{H}_{\alpha, p} \, \dmu}$ is a bounded operator  for each  $\alpha \in \Lambda$.   Now we show that $T_\alpha$ is decomposable. Let $f_\alpha \in \text{L}^\infty(X, \mu)$ and $u \in \int^\oplus_{X} \mathcal{H}_{\alpha,p}  \, \dmu$. Then 
\begin{align*}
\left (T_\alpha \left (  \int^\oplus_{X} f_\alpha(p) \cdot \mathrm{Id}_{\mathcal{H}_{\alpha,p}}  \, \dmu \right ) \right ) u  &= \left ( T  \left (  \int^{\oplus_\text{loc}}_{X} f_\alpha(p) \cdot \mathrm{Id}_{\mathcal{H}_{\alpha,p}}  \, \dmu \right ) \right ) u \\
&=  \left ( \left ( \int^{\oplus_\text{loc}}_{X} f_\alpha(p) \cdot \mathrm{Id}_{\mathcal{H}_{\alpha,p}}  \, \dmu \right ) T \right ) u \\
&= \left ( \left (  \int^\oplus_{X} f_\alpha(p) \cdot \mathrm{Id}_{\mathcal{H}_{\alpha,p}}  \, \dmu \right ) T_\alpha \right ) u.
\end{align*}
Here,  the function  $f_\alpha \in \text{L}^\infty(X, \mu)$  and the element $u \in \int^\oplus_{X} \mathcal{H}_{\alpha,p}  \, \dmu$ were chosen arbitrarily. This shows that $T_\alpha$ is in the commutant of the class of all diagonalizable bounded operators on $\int^\oplus_{X} \mathcal{H}_{\alpha, p} \, \dmu$. Thus by Theorem \ref{thm;DeDibo vNA},  $T_\alpha$ is decomposable, for every $\alpha \in \Lambda$.
\end{proof}

\begin{theorem} \label{thm; M DEC = M DIAG Commutant}
Let $\big \{ \mathcal{H}_p; \mathcal{E}_p = \{\mathcal{H}_{\alpha, p}\}_{\alpha \in \Lambda}; \mathcal{D}_p \big \}_{p \in X}$  be a family of quantized domains.  Suppose either $\Lambda$ is a countable set or  $\mu$ is a counting measure on $X$, then 
\begin{equation*}
C^\ast_{\mathcal{E}, \text{DEC}} \left(\displaystyle \dilX \mathcal{D}_p \, \dmu \right) = \left (C^\ast_{\mathcal{E}, \text{DIAG}} \left(\displaystyle \dilX \mathcal{D}_p \, \dmu \right)\right)^\prime.
\end{equation*}
\end{theorem}
\begin{proof}
We have $C^\ast_{\mathcal{E}, \text{DEC}} \left(\displaystyle \dilX \mathcal{D}_p \, \dmu \right) \subseteq \left (C^\ast_{\mathcal{E}, \text{DIAG}} \left(\displaystyle \dilX \mathcal{D}_p \, \dmu \right) \right)^\prime$ (see Remark \ref{rem; M Dec and M Diag commutant}). It is enough to show the other containment. Let $T \in \left (C^\ast_{\mathcal{E}, \text{DIAG}} \left(\displaystyle \dilX \mathcal{D}_p \, \dmu \right) \right)^\prime$. Then it follows from  Lemma \ref{lem; V_alphaTV_alpha star is decomposable}  that the restriction  $T_\alpha = T\big |_{\int^\oplus_{X} \mathcal{H}_{\alpha, p} \, \dmu}$ is a bounded decomposable operator for each $\alpha \in \Lambda$. Thus for each fixed $\alpha \in \Lambda$, by following  \ref{def;Decbo} of Definition \ref{def;Debo} we get a family $\big \{ S_{\alpha, p} \in \mathcal{B} \big ( \mathcal{H}_{\alpha, p} \big ) \big \}_{p \in X}$ of bounded operators such that 
\begin{equation} \label{eqn; T alpha}
T_\alpha = \int^\oplus_{X} S_{\alpha, p} \, \dmu.
\end{equation}
Since $T$ is locally bounded, $T_\beta \big |_{\int^\oplus_{X} \mathcal{H}_{\alpha, p} \, \dmu} = T_\alpha$, whenever $\alpha \leq \beta$. So, in view of Equation \eqref{eqn; T alpha}, we have $S_{\beta, p} \big |_{\mathcal{H}_{\alpha, p}} = S_{\alpha, p}$ for a.e. $p \in X$. That is, there exists a measurable set, say $E_{\alpha, \beta} \subseteq X$ such that $\mu(E_{\alpha, \beta}) = 0$ and $S_{\beta, p} \big |_{\mathcal{H}_{\alpha, p}} \neq S_{\alpha, p}$ for every $p \in E_{\alpha, \beta}$. Let $E_\alpha := \bigcup\limits_{\beta \in \Lambda, \alpha \leq \beta} E_{\alpha, \beta}$ and $E :=  \bigcup\limits_{\alpha \in \Lambda} E_{\alpha}$. 

Now assume that $\Lambda$ is countable, then $E_\alpha$ is a measurable set for each $\alpha \in \Lambda$. This implies that $E$ is also a measurable set, and, moreover, $\mu(E) = 0$. On the other hand, if $\mu$ is a counting measure on $X$, then for any $\alpha \leq \beta$, the set $E_{\alpha, \beta} = \emptyset$ and thus $E = \emptyset$. So in both cases (that is, when $\Lambda$ is countable or $\mu$ is a counting measure on $X$), we consider the set $X \setminus E$ and without loss of generality we again denote it by $X$. It follows that $S_{\beta, p} \big |_{\mathcal{H}_{\alpha, p}} = S_{\alpha, p}$ for all $p \in X$, whenever $\alpha \leq \beta$. This yields a locally bounded operator $T_p : \mathcal{D}_p \rightarrow \mathcal{D}_p$ (see Equation \eqref{eq; inverese limit of bounded operators}) defined as 
\begin{equation} \label{eqn; Tp}
T_p := \varprojlim\limits_{\alpha \in \Lambda} S_{\alpha, p}, \; \; \; \text{for each} \; \; p \in X.  
\end{equation}
Now we show that $T = \displaystyle \dilX T_p \, \dmu$. Let $u \in \displaystyle \dilX \mathcal{D}_p \, \dmu$. Then by Definition \ref{Defn: directint_loc},  $u \in \int^\oplus_{X} \mathcal{H}_{\alpha, p} \, \dmu$ for some $\alpha \in \Lambda$. Since $T$ is locally bounded, $Tu \in \int^\oplus_{X} \mathcal{H}_{\alpha, p} \, \dmu$ and it follows from Equations \eqref{eqn; T alpha} and \eqref{eqn; Tp} that 
\begin{equation*}
\big (Tu \big )(p) = \big (T_\alpha u \big )(p) = S_{\alpha, p} u(p) =  T_p u(p) \; \; \; \text{for a.e.} \; \; p \in X.  
\end{equation*}
Since $u \in \displaystyle \dilX \mathcal{D}_p \, \dmu$ was chosen arbitrarily, by following  \ref{def;Dec(lbo)} of Definition \ref{def;DecDiag(lbo)}, we obtain 
$$T = \displaystyle \dilX T_p \, \dmu.$$ Thus $T \in C^\ast_{\mathcal{E}, \text{DEC}} \left(\displaystyle \dilX \mathcal{D}_p \, \dmu \right)$.  This proves the result. 
\end{proof}

\subsection*{Acknowledgment}
The first named author kindly acknowledges the financial support received as an Institute postdoctoral fellowship from the Indian Institute of Science Education and Research Mohali. The first named author also sincerely thanks his Ph.D. thesis supervisor for introducing him to the topic of direct integral and disintegration of Hilbert spaces. The second named author would like to thank SERB (India) for a financial support in the form of a Startup Research Grant (File No. SRG/2022/001795). The authors express their sincere thanks to DST for a financial support in the form of the FIST grant (File No. SR/FST/MS-I/2019/46(C)) and the Department of Mathematical Sciences, IISER Mohali, for providing the necessary facilities to carry out this work. 

\subsection*{Declaration} The authors declare that there are no conflicts of interest.

\end{document}